\documentclass[preprint,12pt,authoryear]{elsarticle}

\journal{Physica D}

\usepackage
{
    graphicx,
    amssymb,
    amsmath,
    amsthm,
    xcolor,
    dsfont,
    algpseudocode,
    tikz,
    mathtools,
    nicefrac,
    nicematrix,
    tabularray,
    diagbox,
    booktabs,
    algorithm2e,
    adjustbox,
    array
}

\usepackage[utf8]{inputenc}
\usetikzlibrary{decorations.pathmorphing,patterns}

\usepackage[ngerman, english]{babel}

\usepackage[
left=25mm,
right=25mm,
top=20mm,
bottom=35mm
]{geometry}

\usepackage{subcaption}

\usepackage[colorlinks,
            pdffitwindow=false,
            plainpages=false,
            pdfpagelabels=true,
            pdfpagemode=UseOutlines,
            pdfpagelayout=SinglePage,
            bookmarks=false,
            colorlinks=true,
            hyperfootnotes=false,
            linkcolor=blue,
            citecolor=green!50!black]
{hyperref}

\usepackage{bm}
\usepackage{cleveref}

\usepackage{enumitem}

\usepackage{tabularx}

\usepackage{todonotes}
\usepackage[normalem]{ulem}

\newcommand{\Z}{\mathbb Z}
\newcommand{\R}{\mathbb{R}}
\newcommand{\N}{\mathbb{N}}
\newcommand{\C}{\mathbb{C}}

\newcommand{\Ncal}{\mathcal{N}}

\newcommand{\Ptest}{P_{\mathrm{test}}}

\newcommand{\Fourier}[1]{\widehat{#1}}

\DeclareMathOperator*{\argmin}{arg\,min}

\newcommand{\lossfull}{\ell}

\newcommand\Ic{\mathcal I}
\newcommand\Jc{\mathcal J}
\newcommand\Kc{\mathcal K}
\newcommand\Ac{\mathcal A}
\newcommand\Gc{\mathcal G}
\newcommand\Oc{\mathcal O}
\newcommand\Yc{\mathcal Y}
\newcommand\Xc{\mathcal X}
\newcommand\Sc{\mathcal S}
\newcommand\Dc{\mathcal D}
\newcommand\Koopman{\Kc}

\newtheorem{theorem}{Theorem}[section]
\newtheorem{lemma}{Lemma}%
\newtheorem{proposition}{Proposition}%
\newtheorem{remark}{Remark}%
\newtheorem{example}{Example}[section]%

\begin{document}

\begin{frontmatter}

\title{Group-Convolutional Extended Dynamic Mode Decomposition}

\author[1,6]{Hans Harder}
\author[2]{Feliks Nüske}
\author[3]{Friedrich M. Philipp}
\author[4]{Manuel Schaller}
\author[3]{Karl Worthmann}
\author[5,6]{Sebastian Peitz}

\affiliation[1]{
    organization={Department of Computer Science, Paderborn University},
    city={Paderborn}, 
    country={Germany}
}
\affiliation[2]{
    organization={Max Planck Institute for Dynamics of Complex Technical Systems},
    city={Magdeburg}, 
    country={Germany}
}
\affiliation[3]{
    organization={Institute of Mathematics, Technische Universität Ilmenau},
    city={Ilmenau}, 
    country={Germany}
}
\affiliation[4]{
    organization={Faculty of Mathematics, Chemnitz University of Technology}, 
    city={Chemnitz}, 
    country={Germany}
}
\affiliation[5]{
    organization={Department of Computer Science, TU Dortmund}, 
    city={Dortmund}, 
    country={Germany}
}
\affiliation[6]{
    organization={Lamarr Institute for Machine Learning and Artificial Intelligence},
    city={Dortmund}, 
    country={Germany}
}

\begin{abstract}
This paper explores the integration of symmetries into the Koopman-operator framework for the analysis and efficient learning of equivariant dynamical systems using a group-convolutio\-nal approach. Approximating the Koopman operator by finite-dimensional surrogates, e.g., via extended dynamic mode decomposition (EDMD), is challenging for high-dimensional systems due to computational constraints. 
To tackle this problem with a particular focus on EDMD, we demonstrate---under suitable equivarance assumptions on the system and the observables---that the optimal EDMD matrix is %
equivariant. That is, its action on states can be described by group convolutions and the generalized Fourier transform.
We show that this %
structural property
has many advantages for equivariant systems, in particular, that it allows for data-efficient learning, fast predictions and fast eigenfunction approximations.
We conduct numerical experiments on the Kuramoto--Sivashinsky equation and a $\lambda$-$\omega$ spiraling wave system, both nonlinear partial differential equations, providing evidence of the effectiveness of this approach, and highlighting its potential for broader applications in dynamical systems analysis.
\end{abstract}

\end{frontmatter}

\section{Introduction}

Dynamical systems are fundamental to the modeling of natural phenomena and engineering systems, which are often modeled by differential equations and describe, among others, mechanical systems, the diffusion of heat in a material, the propagation of sound waves in air, or the flow of fluids in pipes and oceans. 
These equations are indispensable in fields ranging from physics and engineering to finance and biology, making them key problem-solving tools.

Linear systems are particularly appealing due to their well-understood properties and the powerful mathematical techniques available for their analysis, see, e.g., \cite{trefethen2022numerical}. For example, one can characterize the behavior of these systems by analyzing their eigenvalues and eigenvectors or derive non-parametric representations using input-output data, see, e.g., \cite{willems2005note,van2020willems}. 
In case of nonlinear dynamics, the Koopman operator, see \cite{Koo31,RMB+09}, provides a transformation into a linear framework by focusing on the evolution of observables---functions of the system state---rather than the state itself. This linearization can reveal underlying structures and long-term behavior. 

Still, there are only few systems for which the Koopman operator has a finite-dimensional representation, requiring an approximation by means of, e.g., extended dynamic mode decomposition (EDMD), see, e.g., the survey article~\cite{brunton2022modern}, the collection~\cite{mauroy2020koopman} and the references therein.

EDMD constructs a matrix that can contain too many entries for efficient computations.
This can happen when the state dimension is already very large, as is usually the case for discretizations of partial differential equations (PDEs). To address this issue, one can often exploit system-inherent properties like, e.g., symmetries. 
Accordingly, symmetry-exploiting methods have become more popular in recent years, even beyond the scope of dynamical systems: %
For example, convolutional neural networks and more generally geometric deep-learning methods \cite{BBCV21} are designed to respect task-inherent equivariances. 
Constructing learning algorithms that respect these equivariances can increase performance and significantly alleviate requirements on training data, see, e.g., \cite{MBB+15,CW16,BBL+17,AGS21,BBCV21,HRV+24}; \cite{dierkes2023hamiltonian}.
The presence of symmetry can lead to a decomposition of the EDMD matrix into simpler, more tractable components. It appears that hitherto proposed symmetry-based methods can be split into two categories: Some methods enforce block-wise diagonal structures in the EDMD matrix, see \cite{sinha2020koopman,PHN+24}, and some methods enforce diagonal structures in Fourier space. Our contribution falls into the latter category and is related to the works of \cite{salovaKoopmanOperatorIts2019,baddooPhysicsinformedDynamicMode2023,hochrainerApproximationTranslationInvariant2024b,nathan2018applied}. Here, the work of \cite{salovaKoopmanOperatorIts2019} provides a more general group-theoretic treatment and, among other things, discusses the importance of choosing appropriate symmetry-respecting observables.
Another focus is on efficient methods for translation-equivariant dynamical systems, see \cite{baddooPhysicsinformedDynamicMode2023,hochrainerApproximationTranslationInvariant2024b} (the work of \cite{baddooPhysicsinformedDynamicMode2023} also considers different aspects, but their proposed method for translational equivariance is most relevant for our purposes). 
In particular, both \cite{baddooPhysicsinformedDynamicMode2023} and \cite{hochrainerApproximationTranslationInvariant2024b} discuss how the EDMD matrix can be forced to be (block-)diagonal in Fourier space for PDEs. Finally, \cite{nathan2018applied} points out the relationship between linearizing transformations such as the Cole-Hopf transform on the one hand and appropriate choices of observables on the other, in conjunction with the Fourier transform.

\paragraph{Contributions and organization} 

Our approach extends and generalizes these developments into multiple directions: Firstly, we assume a more general group-theoretic setting.
In this regard, we show how one can enforce an equivariant structure in the EDMD matrix and demonstrate that this approach is optimal under suitable %
assumptions on the system dynamics, the observables and the sampling distribution. 
Crucially, this approach does not construct the EDMD matrix explicitly, but represents it implicitly by means of a convolution kernel. Making use of the generalized Fourier transform, we can then compute its eigenvalues and eigenvectors, and gain a significant computational advantage.

To this end, we provide an introduction to the Koopman operator and its approximation via EDMD (\Cref{sec:preliminaries}), followed by a discussion of basic group-theoretic aspects (\Cref{sec:groups,sec:convolutions-and-fourier-transform}). \Cref{sec:equivariance-edmd} contains the main contributions of this work: %
On the one hand, assuming an appropriate choice of the sampling distribution and the observables, we show that equivariance is inherited by the EDMD matrix
and that it can be represented by a group convolution (\Cref{thr:convolution}). %
On the other hand, we show how one can conveniently learn the corresponding convolution kernel in Fourier space using the generalized Fourier transform. And finally, we conduct experiments by means of %
the two-dimensional Kuramoto--Sivashinsky equation %
to demonstrate the advantages of our proposed %
group-convolutional approach in the low-data regime (see \Cref{sec:experiments}).

\paragraph{Notation}
We will frequently require sets of the form $\C^{\Ic}$, $\C^{\Ic \times \Jc}$, $\C^{\Ic \times \Jc \times \Dc}$, etc., where $\Ic$, $\Jc$ and $\Dc$ are finite sets. The intention is to think of the elements in $\C^{\Ic}$ as vectors instead of functions, similarly we think of elements in $\C^{\Ic \times \Jc}$ as matrices. The individual entries of $A \in \C^{\Ic \times \Jc}$ are denoted by $A_{i,j}$ for $(i,j) \in \Ic \times \Jc$.
Vector-vector or matrix-vector products are defined accordingly. The main advantage is that this simplifies indexing: For example, if $\Jc = \Jc_1 \times \Jc_2$, we can easily define the matrix-vector product of $A \in \C^{\Ic \times \Jc}$ with $x \in \C^{\Jc}$ by $A x \in \C^{\Ic}$, $(A x)_i = \sum_{(j_1,j_2) \in \Jc} A_{i,(j_1,j_2)} x_{j_1,j_2}$. Throughout this text, we denote random variables in boldface, e.g., $\bm y, \bm u, \bm u^+$. Given a probability space $(\Yc,\Sigma,P)$ with $\bm y \sim P$ and some measurable function $f : \Yc \rightarrow \R$, the expected value is defined as $\mathbb E[f(\bm y)] \coloneqq \int f(y) P(dy)$.

\paragraph{Code and resources} The code for our experiments in \Cref{sec:experiments} and some supplementary material for \Cref{ex:spring-system} can be found on GitHub, \begin{center}
   \url{https://github.com/graps1/convolutional-edmd}. 
\end{center} 
For the case of \Cref{ex:spring-system}, our codebase contains an implementation that models the system dynamics using a group-convolutional approach. For the experiments on the Kuramoto-Sivashinsky equation in \Cref{sec:experiments}, we provide a readable implementation that mimics the mathematical descriptions used in this paper.

\section{Preliminaries: The Koopman operator, EDMD, and symmetries}\label{sec:preliminaries}

The Koopman operator evolves functions (called observables) forward in time by composing them with a given dynamical system. 
For such a system $\Phi : \Yc \rightarrow \Yc$ on some state space~$\Yc$, one defines the Koopman operator $\Koopman$ on observables of the form $f : \Yc \rightarrow \C$ via the composition %
$$
    \Koopman f \coloneqq f \circ \Phi. 
$$
If $f$ is a vector-valued function $f: \Yc \rightarrow \C^m$, then $\Koopman f$ is defined analogously. The Koopman operator is linear by definition, even though the underlying system dynamics may be nonlinear, but it is infinite-dimensional in general. By studying properties of the Koopman operator, such as its eigenvalues and eigenfunctions, one can obtain insights into the underlying system dynamics.
However, due to the infinite number of dimensions, one has to resort to finite-dimensional approximations using extended dynamic mode decomposition (EDMD), see, e.g., \cite{WKR15}.
In EDMD, one projects $\Koopman$ onto a set of \emph{base observables}, whose evolution is described by a finite linear system. This system is found by drawing states from some distribution and solving a certain least-squares problem. We need the following additional assumptions:
\begin{itemize}[noitemsep]
    \item In order to draw samples from $\Yc$, we assume that it is part of a probability space $(\Yc,\Sigma,P)$, where $\Sigma \subseteq 2^\Yc$ is a $\sigma$-algebra of events and $P : \Sigma \rightarrow [0,1]$ a probability measure. We use  $\bm y$ to denote a random variable with distribution $P$ ($\bm y \sim P$).
    \item Additionally, if the Koopman operator is applied to square-integrable observables, it must again produce square-integrable observables. In other words, $f \in L^2_P(\Yc)$ implies $\Koopman f \in L^2_P(\Yc)$ so that the EDMD loss in \eqref{eq:edmd-loss} below has a well-defined expectation.
    \item The dictionary of base observables is given by a set of predefined $\psi_1,\dots,\psi_m \in L^2_P(\Yc)$, or, more simply, $\psi \in L^2_P(\Yc)^m$. 
\end{itemize}
Extended dynamic mode decomposition now tries to find the matrix $K$ that solves the optimization problem 
\begin{align}
    \min_{K \in \C^{m \times m}} \mathbb E[ \lVert K \psi(\bm y) - \Koopman \psi( \bm y) \rVert^2 ].
    \label{eq:edmd-loss}
\end{align}
The solution can be seen as a Galerkin projection of the Koopman operator, see \cite{WKR15}, and the \emph{finite-data approximation} to it is conventionally called the \emph{EDMD matrix}. However, since we are dealing with finitely many samples only at the very end, we will call $K$ the \emph{EDMD matrix} as well, as it is the central object that EDMD tries to learn.
Intuitively, $K$ defines a linear map in $\C^m$ that mimics the Koopman operator's action on the observables $\psi_1,\dots,\psi_m$.
The corresponding EDMD \emph{operator} %
is defined on the span of the base observables, $f = c^T \psi \in \mathrm{span}\, \{\psi_1,\dots,\psi_m\}$, with $c \in \C^m$, via the mapping $$
    (\Koopman_{\text{EDMD}} f)(y) \coloneqq c^T K \psi(y).
$$ 
Under some additional assumptions, it can be shown that $\Koopman_{\text{EDMD}}$ converges to the true Koopman operator when $m \rightarrow \infty$, see \cite{KM18b} and \cite{klus2016convergence} and also the very general analysis of transfer-operator learning provided in \cite{mollenhauer2022kernel,kostic2024sharp,kostic2024consistent}. Besides the projection error, one may deviate from the true EDMD matrix if the expectation in \eqref{eq:edmd-loss} is approximated using finitely many samples. The resulting finite-data error term is well studied and can be bounded, see, e.g., {\cite{Mez22,ZZ22,NPP+23}, \cite{PSW+24} for kernel EDMD, in case of either i.i.d.\ or ergodic sampling and the recent preprint \cite{llamazares2024data} with particular focus on noisy data.
Moreover, uniform error bounds on the approximation error were recently derived in~\cite{KPS+24} based on Wendland radial basis functions and in \cite{yadav2025approximation} using Bernstein polynomials.

Most Koopman-theoretic investigations have hitherto considered ordinary or stochastic differential equations. More general settings can be found in \cite{baddooPhysicsinformedDynamicMode2023, hochrainerApproximationTranslationInvariant2024b, PSB+24,nathan2018applied}, which cover PDEs as well.
Due to the usually very high state dimension in the PDE case stemming from spatial discretization, it is prohibitive to learn a full-scale EDMD matrix and advantageous to apply more efficient approaches. %
Previous works such as \cite{baddooPhysicsinformedDynamicMode2023,hochrainerApproximationTranslationInvariant2024b} already indicated that one may exploit the Fourier transform using translational symmetries.
In general, symmetries occur in a large number of PDE systems implying that the dynamics are equivariant under certain group actions, which depend on the system dynamics as well as on the domain and the boundary conditions.
One example is %
the Kuramoto--Sivashinsky equation, which can be written in coordinate-free form as
\begin{equation}
    \dot y = \Ncal(y) = -\Delta y - \Delta^2 y - \textstyle\frac 1 2 \lVert \nabla y \rVert^2, \label{eq:ks}
\end{equation}
where $y$ is an underlying time-dependent scalar field, $\Delta$ the Laplace operator, $\Delta^2$ the biharmonic operator and $\lVert \nabla y \rVert^2$ the squared length of $y$'s time-dependent gradient field. The resulting flow $\Phi$ can be obtained by solving the PDE for a fixed time step, say $\Phi(y(t)) = y(t+\delta)$ for some $\delta >0$.
This particular PDE is equivariant under any transformation of the Euclidean group. %
This means that shifting, rotating or reflecting the solution $y$ and then applying $\Ncal$ yields the same result as shifting, rotating or reflecting $\Ncal(y)$, see, e.g., \cite{HRV+24}. In particular, shifted, rotated or reflected solutions of the PDE are again solutions. %
Here, we utilize this group-theoretic perspective %
to leverage the generalized Fourier transform for efficient learning.  %

\subsection{Groups and group actions} \label{sec:groups}

The following discussion of group-theoretic aspects is based on \cite{ahlanderApplicationsGeneralizedFourier2005,ahlander_eigenvalues_2006} and \cite{terras1999fourier}.

A \emph{group} is a tuple $\Gc = (G,\circ)$ consisting of a set $G$ and an associative operation $\circ: G \times G \rightarrow G, (g,h) \mapsto gh$ that contains an identity $e$ and inverses. For brevity, we will also use $\Gc$ to refer to the set, i.e., $\Gc = G$.
The group is said to be \emph{abelian} if the composition operation is commutative. Throughout this text, $\Gc$ is assumed to be finite.

\begin{example}
    The symmetry group of the square describes rotations and reflections of a square across its center. It consists of an identity element, rotations of 90, 180 and 270 degrees, and reflections across the diagonal as well as horizontal and vertical axes. The group operation chains these elements, in the sense that two 90 degree rotations amount to one 180 degree rotation.
\end{example}

A \emph{left group action} $\cdot: \Gc \times \Yc \rightarrow \Yc$ of $\Gc$ on a set $\Yc$ is defined by $(g,y) \mapsto g \cdot y$ such that it is compatible with the group operation, i.e., $e \cdot y = y$ and $(g h) \cdot y = g \cdot (h \cdot y)$ for all $g, h \in \Gc$ and $y \in \Yc$. An analogous definition holds for \emph{right} group actions, but in the sense that $(g,y) \mapsto y \cdot g$ with $y \cdot e = y$ and $y \cdot (g h) = (y \cdot g) \cdot h$.
We say that a left resp.\ right group action is \emph{free} if $g \cdot y = y$ resp.\  $y \cdot g = y$ already implies $g = e$ (that is, only the identity element can map elements in $\Yc$ to themselves). Finally, we call a set $\Gc \cdot y = \{ g \cdot y : g \in \Gc \}$ resp.\  $y \cdot \Gc = \{ y \cdot g : g \in \Gc \}$ an \emph{orbit} of a left resp.\ right group action. Notably, the set of orbits forms a partition of $\Yc$. Intuitively, an orbit collects a (maximal) subset of elements in $\Yc$ that can be transformed into each other by application of group operations.

\begin{example}[Orbits]
    \label{ex:orbits}
    Suppose $\Gc = (\Z_2,+)$ is the cyclic group of order two. For $\Yc = \Z_2 \times \Z_2$, the mapping $(y_1,y_2)\cdot g = (y_1 + g, y_2 + g)$ defines a free right group action. This results in the orbits $\{ (0,0), (1,1) \} = (0,0) \cdot \Gc$ and $\{ (1,0), (0,1) \} = (0,1) \cdot \Gc$. An analogous left action can be defined.
\end{example}

If $\Gc$ acts from the left on two sets $\Xc$ and $\Yc$, we say that a function $f : \Xc \rightarrow \Yc$ is \emph{left invariant} if $f(x) = f(g \cdot x)$, and we say that it is \emph{left equivariant} if $g \cdot f(x) = f(g \cdot x)$ holds for all $g \in \Gc$ and $x \in \Xc$. (Note that the spaces $\Xc$ and $\Yc$ might be unrelated, and that the group action might behave differently in both spaces.) If $\Gc$ acts on $\Xc$ and $\Yc$ from the right, the definition is analogous.

\begin{example}[Actions on PDEs]
    Consider the group $(\{+1,-1\}, \cdot)$, which acts on $\R$ via multiplication. We can define a group action on the set of twice-differentiable functions  $y \in C^2(\R)$ by setting $(g \cdot y)(x) = y(- g \cdot x)$, and get $\partial_x (g \cdot y) = g \cdot (-\partial_x y)$. Hence, equivariance for $\Ncal(y) = \partial_{xx}^2 y$ or $\Ncal(y) = (\partial_x y)^2$ holds, but not for $\Ncal(y) = \partial_x y$.
\end{example}

\begin{remark}[Orbits and free actions]\label{rem:orbits-and-free-actions}
    If $\Gc$ acts freely from the right on some finite set $\Oc$, then the set of orbits defines an equivalence relation on~$\Oc$. %
    If $\Sc \subseteq \Oc$ is a complete collection of orbit representatives (that is, $\Sc$ contains exactly one element from every orbit and nothing else), then we have a one-to-one correspondence between $\Oc$ and $\Sc \times \Gc$ via the bijection $\Sc \times \Gc \rightarrow \Oc, (s, g) \mapsto s \cdot g$. This fact allows us to switch back-and-forth between vectors $y \in \C^\Oc$ and vector-valued functions on $\Gc$, namely $y : \Gc \rightarrow \C^\Sc$, via the mapping $y_s(g) \coloneqq y_{s \cdot g}$. We note that this also holds for left group actions. %
\end{remark}

\begin{example}
    Consider $\Oc \coloneqq \Yc = \Z_2 \times \Z_2$ and $\Gc = (\Z_2,+)$ from \Cref{ex:orbits}. A valid set of orbit representatives is $\Sc = \{ (0,0), (1,0) \}$, and we indeed have a one-to-one correspondence between $\Oc$ and $\Sc \times \Gc$ via the identification
    $$
        (0,0) \mapsto ((0,0), 0), \quad (1,1) \mapsto ((0,0), 1),
        \quad (1,0) \mapsto ((1,0), 0), \quad (0,1) \mapsto ((1,0), 1).
    $$
    Intuitively, every $o \in \Oc$ can be uniquely represented as $o = s \cdot g$ for an orbit representative $s \in \Sc$ and a group element $g \in \Gc$, and hence we can simply identify $o$ with $(s, g)$.
\end{example}

Throughout this paper, we assume that $\Gc$ defines a free right-group action on a finite set~$\Oc$. By \Cref{rem:orbits-and-free-actions}, there is a set of orbit representatives $\Sc \subseteq \Oc$ such that $\Sc \times \Gc$ and $\Oc$ are in one-to-one correspondence, and that a bijection is defined by $(s,g) \mapsto s \cdot g$.

\subsection{Equivariant matrices}\label{sec:equivariant-matrices}

We call a matrix $K \in \C^{\Oc \times \Oc}$ \emph{equivariant} if it satisfies $K_{o, o'} = K_{o \cdot g, o' \cdot g}$ for all $g \in \Gc$ and all $o,o'\in \Oc$. %
Matrices of this type are used to model equivariant structures in dynamical systems.
As an important property, matrix-vector multiplications involving %
equivariant matrices can be efficiently evaluated %
by \emph{group convolutions}. Here, in view of symmetry-exploiting learning, we utilize the inherent connection of equivariant matrices with convolution kernels in two ways: First, we reduce the storage requirements, as the convolution kernel stores fewer elements than the corresponding matrix, and second, we achieve a significant speed-up in computations, as %
propagating the system or computing its eigenvalues may be performed by a
generalized version of the Fourier transform. %

To understand the connection between equivariant matrices and equivariant linear functions, \Cref{thr:equivariant-matrices} shows that both are equivalent for group actions that permute indices:

\begin{lemma}
    \label{thr:equivariant-matrices}
    Suppose $\Gc$ acts on $\Oc$ from the right, $K \in \C^{\Oc \times \Oc}$ is a matrix, and a left group action on $\C^\Oc$ is defined via $(g \cdot y)_o \coloneqq y_{o \cdot g^{-1}}$ for all $y \in \C^\Oc$ and $o \in \Oc$. Then the following statements are equivalent:  %
    \begin{enumerate}[noitemsep]
        \item[(i)] $K$ is an equivariant matrix.
        \item[(ii)] The function $f(y) = K y$ is left equivariant.
    \end{enumerate}
\end{lemma}
\begin{proof}
    For the direction $(i) \Rightarrow (ii)$,  pick an arbitrary $o \in \Oc$. Then note that
    \begin{align*}
        ( K (g\cdot y) )_o
        = \sum_{o' \in \Oc} K_{o,o'} (g \cdot y)_{o'}
        &= \sum_{o' \in \Oc} K_{o \cdot g^{-1}, o' \cdot g^{-1}} y_{o' \cdot g^{-1}} \\
        &= \sum_{o' \in \Oc} K_{o \cdot g^{-1}, o'} y_{o'}
        = ( K y)_{o \cdot g^{-1}} = (g \cdot K y)_o.
    \end{align*} 
    It follows that $K (g \cdot y) = g \cdot K y$, and hence $f$ must be equivariant. For the direction $(ii) \Rightarrow (i)$, pick $o,o' \in \Oc$ and $g \in \Gc$ arbitrarily. Suppose that $e_{o'} \in \C^\Oc$ denotes the canoncial unit vector for index $o'$, and note that $g \cdot e_{o'} = e_{o' \cdot g}$. Using equivariance of $f$,
    $$
        K_{o,o'}  
        = (K e_{o'})_o 
        = (g^{-1} \cdot K (g \cdot e_{o'}))_{o}
        = (K e_{o' \cdot g})_{o \cdot g}
        = K_{o\cdot g, o' \cdot g}.
    $$
    It follows that $K$ is an equivariant matrix.
\end{proof}

\begin{example}[Circulant matrices; heat equation]
    To give an instructive example, we consider the one-dimensional heat equation with periodic boundary conditions. After a spatial discretization via finite differences, we obtain a linear equation of the form 
    $$
        \dot y \propto \begin{bmatrix}
            2 & -1 & 0 & 0 & \dots & 0 & -1 \\
            -1 & 2 & -1 & 0 & \dots & 0 & 0 \\
            0 & -1 & 2 & -1 & \dots & 0 & 0 \\
            \vdots & \vdots & \vdots & \vdots & \ddots & \vdots & \vdots \\
            \vdots & \vdots & \vdots & \vdots & \ddots & \vdots & \vdots \\
            -1 & 0 & 0 & 0 & \dots & -1 & 2 \\
        \end{bmatrix} y.
    $$
    This structure is typical for equivariant systems. Indeed, systems that are equivariant with respect to a cyclic group can be described by matrices whose rows are shifted versions of each other -- in that case, they are called \emph{circulant} matrices.
    More formally, one sets $\Oc \coloneqq \Gc \coloneqq \Z_n$ and defines a group action of $\Gc$ on $\Oc$ (that is, itself) simply by applying the group operation. One then may verify that the above-described matrix is indeed equivariant, or simply consider \Cref{ex:block-circulant}.
\end{example}

\begin{example}[Block-circulant matrices]
    \label{ex:block-circulant}
    Suppose $K \in \R^{nm \times nm}$ is a block-circulant matrix,
    \begin{align*}
        K = \begin{bmatrix}
            C_0 & C_{n-1} & \dots & C_1 \\
            C_1 & C_{0} & \dots & C_2 \\
            \vdots & \vdots & \ddots & \vdots \\
            C_{n-1} & C_{n-2} & \dots & C_0 
        \end{bmatrix}
    \end{align*}
    where $C_i \in \R^{m \times m}$, i.e., $K$ consists of blocks of size $m \times m$. Intuitively, the number of orbits $m$ corresponds to the size of each block, while the number of blocks $n$ corresponds to the size of the group.
    To show equivariance of $K$ with respect to $\Gc = (\Z_n,+)$, we first identify $\Oc =  \{1,\dots,m\} \times \Gc$ with $\{1, \dots, nm\}$ using the bijection $\alpha: (i,j) \mapsto j m + i$. A free group action is defined on $\Oc$ by $(i,j) \cdot g = (i, j + g)$, and on $\{1,\dots,nm\}$ by $\alpha(i,j)\cdot g = \alpha((i,j)\cdot g)$. Indeed, one has for $(i,j),(i',j') \in \Oc$,
    $$
        K_{\alpha(i,j), \alpha(i',j')} 
        = K_{jm+i,j'm+i'}
        = (C_{j-j'})_{i,i'}
        = (C_{(j+g)-(j'+g)})_{i,i'}
        = K_{\alpha(i,j)\cdot g, \alpha(i',j')\cdot g}.
    $$
    For the orbit representatives, one can, for example, pick $\Sc = \{ (1,0), \dots, (m,0) \}$. 
    
\end{example}

\begin{figure}[ht]
    \centering
    \scalebox{0.75}{
        \begin{tikzpicture}
            \node[circle,fill=black,inner sep=3mm] (a) at (-0.3,0) {};
            \node[circle,fill=black,inner sep=3mm] (b) at (3.7,0) {};
            \node[circle,fill=black,inner sep=3mm] (c) at (7.7,0) {};
            \node[circle,fill=black,inner sep=3mm] (d) at (11.7,0) {};
            
            \node at (-2,-0.5) {$k$};
            \node at ( 2,-0.5) {$k$};
            \node at ( 6,-0.5) {$k$};
            \node at (10,-0.5) {$k$};
            \node at (14,-0.5) {$k$};

            \draw[<->] (-0.7, 0.8) -- (0.1, 0.8) node[midway,above] {$p_1$};
            \draw[<->] ( 3.3, 0.8) -- (4.1, 0.8) node[midway,above] {$p_2$};
            \draw[<->] ( 7.3, 0.8) -- (8.1, 0.8) node[midway,above] {$p_3$};
            \draw[<->] (11.3, 0.8) -- (12.1, 0.8) node[midway,above] {$p_4$};
            
            \draw[decoration={aspect=0.9, segment length=16, amplitude=2mm,coil},decorate] (-4,0) -- (0,0); 
            \draw[decoration={aspect=0.9, segment length=16, amplitude=2mm,coil},decorate] (0,0) -- (4,0); 
            \draw[decoration={aspect=0.9, segment length=16, amplitude=2mm,coil},decorate] (4,0) -- (8,0); 
            \draw[decoration={aspect=0.9, segment length=16, amplitude=2mm,coil},decorate] (8,0) -- (12,0); 
            \draw[decoration={aspect=0.9, segment length=16, amplitude=2mm,coil},decorate] (12,0) -- (16,0); 
            
            \fill [fill = white] (-4.5,-1) rectangle (-3.8,1);
            \fill [pattern = north east lines] (-4.5,-1) rectangle (-3.8,1);
            \draw[thick] (-3.8,-1) -- (-3.8,1);
            
            \fill [fill = white] (15.5,-1) rectangle (16.2,1);
            \fill [pattern = north east lines] (15.5,-1) rectangle (16.2,1);
            \draw[thick] (15.5,-1) -- (15.5,1);
        \end{tikzpicture}
    }
    \caption{A spring system of four nodes connected in a row with equal stiffness coefficients. This system is equivariant with respect to reflections along the center.}
    \label{fig:spring-system}
\end{figure}

\begin{example}[Symmetric spring system\footnote{Code available online, \url{https://github.com/graps1/convolutional-edmd}.}]
    \label{ex:spring-system}
    Consider the spring system from \Cref{fig:spring-system}, which consists of four nodes that deviate from their resting positions by $p_1(t),p_2(t),p_3(t),p_4(t) \in \R$, $p(t) = [p_1(t), p_2(t), p_3(t),p_4(t)]\in \R^4$. All springs have the same fixed stiffness coefficient $k$ and the nodes have mass $m$. The force acting on the $i$-th node is given by $F_i(p)$, where
    \begin{align*}
         &F_1(p) = - 2 k p_1 + k p_2, \quad
        &&F_4(p) = - 2 k p_4 + k p_3, \\
         &F_2(p) = - 2 k p_2 + k p_1 + k p_3, \quad 
        &&F_3(p) = - 2 k p_3 + k p_2 + k p_4.
    \end{align*}
    From Newton's second law, we have $\ddot p(t) = F(p(t)) / m$. By moving to position-velocity variables $y(t) = (p(t), \dot p(t)) \in \Yc \coloneqq \R^8$, we %
    may consider the first-order differential equation $\dot y = K y$, where
    $$
        \begin{bmatrix}
            \dot p_1 \\
            \dot p_2 \\
            \dot p_3 \\
            \dot p_4 \\
            \ddot p_1 \\
            \ddot p_2 \\
            \ddot p_3 \\
            \ddot p_4 
        \end{bmatrix} = \begin{bNiceMatrix}[columns-width=auto]
            0 & 0 & 0 & 0 & 1 & 0 & 0 & 0 \\
            0 & 0 & 0 & 0 & 0 & 1 & 0 & 0 \\
            0 & 0 & 0 & 0 & 0 & 0 & 1 & 0 \\
            0 & 0 & 0 & 0 & 0 & 0 & 0 & 1 \\
            \nicefrac{-2k}{m} & \nicefrac k m & 0 & 0 & 0 & 0 & 0 & 0 \\
            \nicefrac k m & \nicefrac{-2k}{m} & \nicefrac k m & 0 & 0 & 0 & 0 & 0 \\
            0 & \nicefrac k m & \nicefrac{-2k}{m} & \nicefrac k m & 0 & 0 & 0 & 0 \\
            0 & 0 & \nicefrac k m & \nicefrac{-2k}{m} & 0 & 0 & 0 & 0
        \end{bNiceMatrix} 
        \begin{bmatrix}
            p_1 \\
            p_2 \\
            p_3 \\
            p_4 \\
            \dot p_1 \\
            \dot p_2 \\
            \dot p_3 \\
            \dot p_4 
        \end{bmatrix}
    $$
    This system is equivariant with respect to reflections along the center, that is, when swapping nodes $1$ and $2$ with nodes $4$ and $3$. More formally, the group $\Gc = \{e,a\}, aa = e$, acts on $\Yc$ via
    $$
        a \cdot \begin{bmatrix} 
            p_1 & p_2& p_3& p_4& \dot p_1& \dot p_2& \dot p_3& \dot p_4
        \end{bmatrix}^T = 
        \begin{bmatrix} 
            p_4& p_3& p_2&p_1&\dot p_4& \dot p_3&\dot p_2& \dot p_1
        \end{bmatrix}^T,
    $$
    with $e$ doing nothing. 
    That is, $a$ just reflects the $p$'s and $\dot p$'s along the center. 
    One can then verify that the system satisfies $g \cdot K y = K (g \cdot y)$ for all $g \in \Gc$ and $y \in \Yc$ (this can be checked by simply computing $a \cdot Ky$ and $K (a \cdot y)$ separately). Additionally, mind that this action simply permutes vector indices, and hence $K$ is also an equivariant matrix by means of \Cref{thr:equivariant-matrices}. The corresponding action of $\Gc$ on $\Oc = \{1,\dots,8\}$ would then be given by $e$ doing nothing and
    \begin{equation}
    \begin{split}
        1 \cdot a = 4, \quad 2 \cdot a = 3, \quad 3 \cdot a = 2, \quad 4 \cdot a = 1, \\
        5 \cdot a = 8, \quad 6 \cdot a = 7, \quad 7 \cdot a = 6, \quad 8 \cdot a = 5.
        \label{eq:spring-example-group-action}
    \end{split}
    \end{equation}
\end{example}

\paragraph{Convolution kernels}

The \emph{convolution kernel} $A : \Gc \rightarrow \C^{\Sc \times \Sc}$ that corresponds to an equivariant matrix $K$ is defined by
\begin{equation}
    A_{s,s'}(g) \coloneqq K_{s,s'\cdot g^{-1}} .
    \label{eq:corresponding-convolution-kernel}
\end{equation}
By assumption, $\Gc$ acts freely on $\Oc$, and thus $\Oc$ and $\Sc \times \Gc$ are in one-to-one correspondence via the bijection $(s,g) \mapsto s \cdot g$.
Therefore, the equivariant matrix that corresponds to a kernel $A$ is given by 
\begin{equation}
    \label{eq:corresponding-equivariant-matrix}
    K_{s \cdot g, s' \cdot h} \coloneqq A_{s,s'}(g h^{-1}).
\end{equation}
\Cref{eq:corresponding-equivariant-matrix} defines the inverse of \eqref{eq:corresponding-convolution-kernel} if one sees these equations as defining functions between the set of equivariant matrices and the set of convolution kernels.

The \emph{convolution} of a kernel $A : \Gc \rightarrow \C^{\Sc \times \Sc}$ with $u : \Gc \rightarrow \C^{\Sc}$ is defined by
\begin{equation}
   A * u : \Gc \rightarrow \C^\Sc, \quad g \mapsto \sum_{h \in \Gc} A(h) u(h^{-1} g). 
   \label{eq:definition-convolution}
\end{equation}
By the identification of $\Oc$ with $\Sc \times \Gc$, we can define the convolution of $A$ with a \emph{vector} $u \in \C^\Oc$ by viewing $u$ as a function $u : \Gc \rightarrow \C^\Sc$ defined by $u_s(g) = u_{s \cdot g}$ (see \Cref{rem:orbits-and-free-actions}). In the same sense, one can also view $A * u$ as a vector in $\C^\Oc$. %
If $K$ is $A$'s corresponding equivariant matrix, one always has
\begin{equation}
    \label{eq:matrix-multiplication-equals-convolution}
    K u = A * u.   
\end{equation}
Hence, if one makes use of the generalized Fourier transform, one may gain a computational advantage when it comes to computing matrix-vector products.

\begin{example}[Symmetric spring system: Convolutions]
    \label{ex:spring-system-convolutions}
    From \Cref{ex:spring-system}, recall that the state can be %
     viewed as an element of $\R^\Oc$, where $\Oc = \{1,\dots,8\}$. We have defined a group action of $\Gc = \{e,a\}$ on $\Oc$ via \eqref{eq:spring-example-group-action}. The corresponding set of orbits is given by
    $$
        \{ 1, 4 \}, \{ 2, 3 \}, \{5, 8\}, \{6,7 \}, 
    $$
    and we may pick the orbit representatives $\Sc = \{1,2,5,6\}$. 
    The identification of $y \in \R^\Oc$ with $y : \Gc \rightarrow \R^\Sc$ according to \Cref{rem:orbits-and-free-actions} is given by
    \begin{align*}
        \begin{NiceArray}{c|cccccccc}
              & p_1 & p_2 & p_3 & p_4 & \dot p_1 & \dot p_2 & \dot p_3 & \dot p_4 \\ 
             \hline
             y \in \R^O & y_1 & y_2 & y_3 & y_4 &y_5 & y_6 & y_7 & y_8 \\ 
             y : \Gc \rightarrow \R^S & y_1(e) & y_2(e) & y_2(a) & y_1(a) &y_5(e) & y_6(e) & y_6(a) & y_5(a)
        \end{NiceArray}
    \end{align*}
    Finally, let us consider the convolution kernel for the matrix $K$. To simplify notation, we sort $\Sc$ in ascending order, and view $y(e),y(a) \in \R^\Sc$ as vectors in $\R^{|\Sc|}$, 
    \begin{align*}
       y(e) = \begin{bmatrix} p_1& p_2& \dot p_1& \dot p_2 \end{bmatrix}^T,
       \quad 
       y(a) = \begin{bmatrix} p_4& p_3& \dot p_4& \dot p_3 \end{bmatrix}^T.
    \end{align*}
    From \cref{eq:matrix-multiplication-equals-convolution} together with \eqref{eq:definition-convolution}, we can derive
    \begin{align}
        \begin{bmatrix} 
            \dot y(e)  \\
            \dot y(a)
        \end{bmatrix}
        = \begin{bmatrix}
            (A * y)(e) \\
            (A * y)(a)
        \end{bmatrix} 
        = \begin{bmatrix}
            A(e) & A(a) \\
            A(a) & A(e) 
        \end{bmatrix}
        \begin{bmatrix}
            y(e) \\
            y(a)
        \end{bmatrix}.
        \label{eq:spring-example-block-convolutional-form}
    \end{align}
    Here, we view $A(e),A(a) \in \R^{\Sc \times \Sc}$ as matrices in $\R^{|\Sc|\times |\Sc|}$. From \cref{eq:corresponding-convolution-kernel} and the entries of $K$ in \Cref{ex:spring-system}, we can derive their values by inspection:
    \begin{align*}
        A(e) = \begin{bNiceMatrix}[columns-width=auto]
            0 & 0 & 1 & 0 \\
            0 & 0 & 0 & 1 \\
            \nicefrac{-2k}{m} & \nicefrac k m & 0 & 0 \\
            \nicefrac k m & \nicefrac{-2k}{m} & 0 & 0     
        \end{bNiceMatrix},
        \quad 
        A(a) = \begin{bNiceMatrix}[columns-width=auto]
            0 & 0 & 0 & 0 \\
            0 & 0 & 0 & 0 \\
            0 & 0 & 0 & 0 \\
            0 & \nicefrac k m & 0 & 0 
        \end{bNiceMatrix}.
    \end{align*}
    Inserting these into \eqref{eq:spring-example-block-convolutional-form}, we obtain the following system of equations:
    \begin{align*}
        \begin{bmatrix}
            \dot p_1 \\
            \dot p_2 \\
            \ddot p_1 \\
            \ddot p_2 \\
            \dot p_4 \\
            \dot p_3 \\
            \ddot p_4 \\
            \ddot p_3
        \end{bmatrix} 
        = 
        \begin{bNiceArray}[columns-width=auto]{cccc|cccc}
            0 & 0 & 1 & 0 & 0 & 0 & 0 & 0 \\
            0 & 0 & 0 & 1 & 0 & 0 & 0 & 0 \\
            \nicefrac{-2k}{m} & \nicefrac k m & 0 & 0 & 0 & 0 & 0 & 0 \\
            \nicefrac k m & \nicefrac{-2k}{m} & 0 & 0 & 0 & \nicefrac k m & 0 & 0 \\
            \hline
            0 & 0 & 0 & 0 & 0 & 0 & 1 & 0 \\
            0 & 0 & 0 & 0 & 0 & 0 & 0 & 1 \\
            0 & 0 & 0 & 0             & \nicefrac{-2k}{m} & \nicefrac k m & 0 & 0 \\
            0 & \nicefrac k m & 0 & 0 & \nicefrac k m & \nicefrac{-2k}{m} & 0 & 0 \\
        \end{bNiceArray}
        \begin{bmatrix}
            p_1 \\
            p_2 \\
            \dot p_1 \\
            \dot p_2 \\
            p_4 \\
            p_3 \\
            \dot p_4 \\
            \dot p_3
        \end{bmatrix} 
    \end{align*}
    Notice that %
    this yields
    the same solution as the original system. However, by re-ordering the variables, we have %
    achieved
    a block-circulant format (compare \cref{eq:spring-example-block-convolutional-form} with \Cref{ex:block-circulant} for the case $n=2$), which is due to the isomorphism of $\Gc$ with $(\Z_2,+)$. 
    Therefore, instead of solving the equation $\dot y = K y$ by explicit construction of $K$, one can compute convolutions using $A$, if it is known. Notice also that we do not need to store the full $8 \times 8$ matrix $K$ (ignoring its sparsity), but can rather store two $4 \times 4$ matrices, namely $A(e)$ and $A(a)$. 
\end{example}

In the next section, we show how the Fourier transform of the convolution kernel can be used to speed up linear operations.

\subsection{The generalized Fourier transform}\label{sec:convolutions-and-fourier-transform}

Before introducing the generalized Fourier transform, we require the notion of a \emph{group representation}. Essentially, a representation maps every group element to a matrix that describes the action of this element on points in space. In case of abelian groups, this notion simplifies significantly, but we are also interested in the more general case.

\paragraph{Group representations}
A \emph{representation} of a finite group $\Gc$ is a function $\rho : \Gc \rightarrow \C^{d_\rho \times d_\rho}$, where $d_\rho \in \N$ is the \emph{degree} of $\rho$, such that $\rho(g)$ is invertible for all $g\in \Gc$ and compatibility with the underlying group is ensured: for $h, g \in \Gc$, one has $\rho(h)\rho(g) = \rho(h g)$. 
Two representations $\rho,\sigma$ with $d_\rho = d_\sigma = d$ are \emph{equivalent} if there is an invertible $T \in \C^{d \times d}$ such that $\rho(g) = T^{-1} \sigma(g) T$ for all $g \in \Gc$.
A representation $\rho$ is called \emph{reducible} if there is an equivalent representation $\sigma$ such that all $\{ \sigma(g) : g \in \Gc\}$ have the same block-diagonal structure.
Finally, there exists a finite set of \emph{irreducible} representations that are pairwise inequivalent. This set is denoted by $\Fourier{\Gc}$, and it satisfies $\sum_{\rho \in \Fourier{\Gc}} d_\rho^2 = |\Gc|$.

\begin{example}[Cyclic group]
    \label{ex:cyclic-group-representation}
    The cyclic group $\Z_3$ has three irreducible representations: $\rho_0,\rho_1,\rho_2 : \Gc \rightarrow \C$, defined by $\rho_k(g) = \exp(-2 \pi \mathrm i g k / 3)$ for $k = 0,1,2$.
\end{example}

\begin{example}[Permutation group]
    The permutation group of three elements can be defined by the set of bijective functions from $\{1,2,3\}$ to itself together with function composition. A representation in $\R^3$ is given for $\sigma : \{1,2,3\} \rightarrow \{1,2,3\}$ by $\rho(\sigma) \in \R^{3 \times 3}$, with $\rho(\sigma)_{ij} = 1$ if $i = \sigma(j)$ and $0$ otherwise. That is, by the set of permutation matrices in $\R^3$.
\end{example}

\paragraph{The generalized Fourier transform}
For a representation $\rho$, the generalized Fourier transform is defined on $u : \Gc \rightarrow \C$ by
\begin{equation}
    \Fourier{ u }(\rho) = \sum_{g \in \Gc} u(g) \rho(g) \in \C^{d_\rho \times d_\rho}.
    \label{eq:generalized-fourier-transform}
\end{equation}
That is, $\Fourier{u}$ takes a representation $\rho$ and maps it to a complex \emph{matrix} of size $d_\rho \times d_\rho$. The inverse of the generalized Fourier transfom is given by 
\begin{equation}
    u(g) = \frac{1}{|\Gc|} \sum_{\rho \in \Fourier{\Gc}} d_\rho \mathrm{Tr}(\rho(g^{-1}) \Fourier{u}(\rho)).
    \label{eq:generalized-inv-fourier-transform}
\end{equation}
This is equivalent to the standard Fourier transform for the case of cyclic groups (see the part about abelian groups at the end of this section).

\begin{remark}
    For abelian groups, fast algorithms for the Fourier transform are available, which need $O(|\Gc| \log |\Gc|)$ operations. For more general groups, fast algorithms are at least conjectured, see \cite{maslen2001cooley}, and already known for a few groups. Either way, it is always possible to compute the GFT in a naive way in $O(|\Gc| \sum_{\rho \in \Fourier{\Gc}} d_\rho^2 ) = O(|\Gc|^2)$.
\end{remark}

The \emph{convolution theorem} connecting Fourier transformations and convolutions can be generalized too; for $u, v:  \Gc\rightarrow \C$, one has
$
    \Fourier{u * v}(\rho) = \Fourier{u}(\rho) \Fourier{v}(\rho)
$
(see, for example, \cite{terras1999fourier}, Theorem 4, p.\ 261).\footnote{Here, the convolution of $u$ with $v$ is defined as in \eqref{eq:definition-convolution}, resulting in a function $u * v : \Gc \rightarrow \C$.}
For the general case of matrices and vectors, let $A : \Gc \rightarrow \C^{\Sc \times \Sc}$ and $u : \Gc\rightarrow \C^\Sc$ and introduce:
\begin{align*}
    &\Fourier{A}(\rho) \in \C^{(\Sc \times \Dc_\rho) \times (\Sc \times \Dc_\rho)},
    \quad &&\Fourier{A}_{(s,i),(s',j)}(\rho) \coloneqq (\Fourier{A_{s,s'}})_{i,j}(\rho), \\
    &\Fourier{u}(\rho) \in \C^{(\Sc \times \Dc_\rho) \times \Dc_\rho}, 
    \quad &&\Fourier{u}_{(s',j), k}(\rho) \coloneqq (\Fourier{u_{s'}})_{j,k}(\rho),
\end{align*}
where $\Dc_\rho = \{1,\dots,d_\rho\}$. 
Intuitively, this way we treat the $\Sc \times \Sc$ resp. $\Sc$ like batch dimensions by defining the Fourier transform individually for every $A_{s,s'}$ and $u_s$. 
Here, 
$\Fourier{A}(\rho)$ and $\Fourier{u}(\rho)$ can also be seen as matrices with shape $|\Sc|d_\rho \times |\Sc|d_\rho$ and $|\Sc| d_\rho \times d_\rho$. 
This is advantageous because we can now extend the convolution theorem to matrix multiplication:
\begin{equation}
    \Fourier{A * u}(\rho) = \Fourier{A}(\rho) \Fourier{u}(\rho). \label{eq:convolution-theorem}
\end{equation}

\paragraph{Eigenvalues and eigenvectors}
As discussed by \cite{ahlander_eigenvalues_2006}, the Fourier transform closely connects the eigenvalues of group convolutions with the eigenvalues of $\Fourier{A}(\rho)$, where $\rho$ ranges over the irreducible representations:

\begin{proposition}[Åhlander and Munthe-Kaas, 2006, Proposition 2] %
    \label{thr:eigenvalues}
    Let $A : \Gc \rightarrow \C^{\Sc \times \Sc}$ be a convolution kernel. Then,
    \begin{enumerate}[noitemsep]
        \item If $\lambda$ is an eigenvalue of $u \mapsto A * u$, then there is $\rho \in \Fourier{\Gc}$ s.t.\ $\lambda$ is an eigenvalue of $\Fourier{A}(\rho)$.
        \item If $\lambda$ is an eigenvalue of $\Fourier{A}(\rho)$ for a $\rho \in \Fourier{\Gc}$, then $\lambda$ is a $d_\rho$-fold %
        eigenvalue of $u \mapsto A * u$.
    \end{enumerate}
\end{proposition}
Besides eigenvalues, \cite{ahlander_eigenvalues_2006} also describe how one can obtain the corresponding eigenvectors, that is, those $v : \Gc \rightarrow \C^\Sc$ that satisfy $A * v = \lambda v$ for some eigenvalue $\lambda$. 
Since $v \mapsto A * v$ is a linear operation on vectors of size $|\Sc||\Gc|$, there are equally many eigenvectors, and one can index them by $(\sigma,i,j)$, where $\sigma \in \Fourier{\Gc}$, $i = 1,\dots,|\Sc|d_\sigma$ and $j = 1,\dots, d_\sigma$ (this follows from $\sum_{\sigma \in \Fourier{\Gc}} d_\sigma^2 = |\Gc|$).
For $(\sigma,i,j)$ fixed, let $(\lambda,\xi)$ be the $i$-th eigenpair of the Fourier-transformed kernel at $\sigma$,
$$
    \Fourier{A}(\sigma) \xi = \lambda \xi.
$$ 
Define $v$ in Fourier space by a zero-padded version of $\xi$, i.e., $\Fourier{v}(\rho) \in \C^{|\Sc| d_\rho \times d_\rho}$ with:
$$
    \Fourier{v}_{\bullet,\ell}(\rho) = \begin{cases}
        \xi & \text{if } \ell = j \text{ and } \rho = \sigma, \\
        0 & \text{else}
    \end{cases}
$$
for all $\rho \in \Fourier{\Gc}$ and $\ell = 1,\dots,d_\rho$.
Note that 
$
    \Fourier{A}(\rho) \Fourier{v}(\rho) = \lambda \Fourier{v}(\rho)
$
holds for all $\rho \in \Fourier{\Gc}$. Using the convolution theorem \eqref{eq:convolution-theorem} and the inverse Fourier transform, one can then see that $(\lambda,v)$ is an eigenpair of $u \mapsto A * u$. By ranging over all $(\sigma,i,j)$, this procedure allows one to construct a full set of eigenvectors.

\begin{remark}[Computational costs]
    To compute a full set of eigenvalues, one therefore needs to solve one eigenvalue problem of size $|\Sc|d_\sigma$ for every $\sigma \in \Fourier{\Gc}$, plus the effort for computing the generalized Fourier transform of $A$. Ignoring the Fourier transform, this results in a complexity of $O(|\Sc|^3 \sum_{\sigma \in \Fourier{\Gc}} d_\sigma^3)$ (if one assumes $O(n^3)$ for an eigenvalue decomposition, where $n$ is the matrix size). In contrast, if one would explicitly construct the equivariant matrix corresponding to $A$, this would require $O(|\Gc|^3|\Sc|^3)$ operations for an eigenvalue decomposition. In case of abelian groups, one has $|\Fourier{\Gc}| = |\Gc|$ and $d_\rho = 1$, which would simplify the complexity in the group-convolutional case to $O(|\Gc||\Sc|^3)$.
\end{remark}

\begin{example}[Symmetric spring system: Fourier transform]
    Let us compute the Fourier-transformed convolution kernel and its eigenvalues from \Cref{ex:spring-system,ex:spring-system-convolutions}. Recall from \Cref{ex:spring-system-convolutions} that $\Gc$ is isomorphic to $\Z_2$. As in \Cref{ex:cyclic-group-representation}, we obtain a set of irreducible representations $\Fourier{\Gc} = \{ \rho_0, \rho_1 \}$ with $d_{\rho_0} = d_{\rho_1} = 1$, defined by
    \begin{align*}
        \rho_0(g) = \begin{cases}
            1 & \text{if } g = e, \\
            1 & \text{if } g = a
        \end{cases}
        \quad\text{and}\quad
        \rho_1(g) = \begin{cases}
            1  &  \text{if } g = e, \\
            -1 & \text{if } g = a.
        \end{cases}
    \end{align*}
    Inserting this into \eqref{eq:generalized-fourier-transform}, the Fourier transform is simply
    \begin{align*}
        \Fourier{u}(\rho) = \begin{cases}
            u(e) + u(a) & \text{if } \rho = \rho_0, \\
            u(e) - u(a) & \text{if } \rho = \rho_1.
        \end{cases}
    \end{align*}
    Hence,
    \begin{align*}
        \Fourier{A}(\rho_0) =  
        \begin{bNiceMatrix}[columns-width=auto]
            0 & 0 & 1 & 0 \\ 
            0 & 0 & 0 & 1 \\ 
            \nicefrac {-2k} m & \nicefrac k m & 0 & 0 \\ 
            \nicefrac k m & \nicefrac {-k} m & 0 & 0 
        \end{bNiceMatrix},\quad
        \Fourier{A}(\rho_1) = 
        \begin{bNiceMatrix}[columns-width=auto]
            0 & 0 & 1 & 0 \\ 
            0 & 0 & 0 & 1 \\ 
            \nicefrac {-2 k} m & \nicefrac k m & 0 & 0 \\ 
            \nicefrac k m & \nicefrac {-3 k} m & 0 & 0 
        \end{bNiceMatrix}.
    \end{align*}
    where we view $\Fourier{A}(\rho_0)$ and $\Fourier{A}(\rho_1)$ as $|\Sc| \times |\Sc|$ matrices. Their eigenvalues are
    \begin{align*}
        &\textstyle
        \Fourier{A}(\rho_0): \quad - i \sqrt{\nicefrac k {2m}} \sqrt{3 + \sqrt{5}},\ 
        + i \sqrt{\nicefrac k {2m}} \sqrt{3 + \sqrt{5}},\ 
        - \sqrt{\nicefrac k {2m}} \sqrt{-3 + \sqrt{5}},\ 
        + \sqrt{\nicefrac k {2m}} \sqrt{-3 + \sqrt{5}}, \\
        &\textstyle
        \Fourier{A}(\rho_1): \quad - i \sqrt{\nicefrac k {2m}} \sqrt{5 + \sqrt{5}},\ 
        + i \sqrt{\nicefrac k {2m}} \sqrt{5 + \sqrt{5}},\ 
        - \sqrt{\nicefrac k {2m}} \sqrt{-5 + \sqrt{5}},\ 
        + \sqrt{\nicefrac k {2m}} \sqrt{-5 + \sqrt{5}},
    \end{align*}
    which match the ones of $K$.
\end{example}

\paragraph{Left eigenpairs}
If $K$ is $A$'s corresponding equivariant matrix, then every eigenpair of $u \mapsto A * u$ is a \emph{right} eigenpair of $K$. \emph{Left} eigenpairs of $K$ correspond to right eigenpairs of $K^T$, and hence we can compute them by introducing a ``transposed'' kernel
$$
    A^T : \Gc \rightarrow \C^{\Sc \times \Sc}, \quad (A^T)_{s,s'}(g) \coloneqq A_{s',s}(g^{-1}),
$$
that computes the matrix-vector product $K^T u = A^T * u$. Therefore, one can construct \emph{left} eigenpairs of $K$ by computing eigenpairs of $u \mapsto A^T * u$ in the above-described way. (This is in particular necessary when it comes to approximating eigenfunctions of the Koopman operator.)
\\
\paragraph{Plancherel's identity} 
This result relates the norm of $u : \Gc \rightarrow \C^\Sc$ with a ``weighted'' norm of its Fourier transform.  Recall that $\Fourier{u}(\rho)$ can be seen as a $|\Sc|d_\rho \times d_\rho$ matrix. From this perspective, the two norms considered here are %
$$
    \lVert u \rVert^2 \coloneqq \sum_{s \in \Sc} \lVert u_s \rVert^2_2
    \quad\text{and}\quad 
    \lVert \Fourier{u} \rVert^2 \coloneqq \frac{1}{|\Gc|} \sum_{\rho \in \Fourier{\Gc}} d_\rho \lVert \Fourier{u}(\rho) \rVert_F^2,
$$
where $\lVert \cdot \rVert_F$ is the Frobenius norm. In particular, if $u : \Gc \rightarrow \C$, Plancherel's identity (see, for example, \cite[Theorem 2, p. 258, exercise on p. 263]{terras1999fourier}) states that 
$
    \lVert u \rVert^2 = \lVert \Fourier{u} \rVert^2.
$
Even further, one can see that it also holds for the case $u : \Gc \rightarrow \C^\Sc$:
$$
    \lVert u \rVert^2 
    = \sum_{s \in \Sc} \frac 1 {|\Gc|} \sum_{\rho \in \Fourier{\Gc}} d_\rho \lVert \Fourier{u_s}(\rho) \rVert^2_F 
    = \frac 1 {|\Gc|} \sum_{\rho \in \Fourier{\Gc}} d_\rho \sum_{s \in \Sc} \lVert \Fourier{u_s}(\rho) \rVert^2_F 
    = \frac 1 {|\Gc|} \sum_{\rho \in \Fourier{\Gc}} d_\rho \lVert \Fourier{u}(\rho) \rVert^2_F = \lVert \Fourier{u} \rVert^2.
$$
We state this result independently:
\begin{lemma}[Plancherel's identity]
    For $u : \Gc \rightarrow \C^\Sc$, one has $\lVert u \rVert^2 = \lVert \Fourier{u} \rVert^2$.
    \label{thr:plancherel}
\end{lemma}

\paragraph{Abelian groups}
If $\Gc$ is abelian, then the situation simplifies significantly. In this case, $\Gc$ can be expressed as the direct sum of cyclic groups, $\Gc = \Z_{p_1} \oplus \dots \oplus \Z_{p_m}$. %
Assuming that we can decompose $\Gc$ this way, %
the order of every irreducible representation $\rho \in \Fourier{\Gc}$ is exactly $d_\rho = 1$, and moreover, we can write $\Fourier{\Gc} = \{ \rho_g : g \in \Gc \}$ with
$$
    \rho_g(h) = e^{-2\pi \mathrm i \sum_{n=1}^m g_n h_n / p_n } \in \C.
$$
The Fourier transform simplifies to:
$$
    \Fourier{u}(\rho_g) = \sum_{h \in \Gc} u(h) e^{-2 \pi \mathrm i \sum_{n=1}^m g_n h_n / p_n},
$$
which is essentially the multi-dimensional discrete Fourier transform. For the case where the transformed object is a matrix or vector-valued function $A : \Gc \rightarrow \C^{\Sc \times \Sc}$ or $u : \Gc \rightarrow \C^{\Sc}$, $\Fourier{A}(\rho)$ is a matrix in $\C^{\Sc \times \Sc}$ and $\Fourier{u}(\rho)$ a vector in $\C^\Sc$. 
Note that \Cref{thr:eigenvalues} also simplifies in the sense that the multiplicity of the eigenvalues vanishes.

\section{Group-convolutional EDMD: Theory and Application} \label{sec:equivariance-edmd}

In this section, we show how equivariant structures can be exploited in EDMD when the observables are appropriately defined. We build on the previously established results from \Cref{sec:groups,sec:equivariant-matrices,sec:convolutions-and-fourier-transform} to show a number of results: The first result concerns the ``optimality'' of convolutions (\Cref{thr:convolution}), that is, we show that there is indeed an optimal equivariant EDMD matrix under certain conditions. Next, we show how one can construct appropriate observables, how one can learn the convolution kernel, and how one can approximate eigenfunctions of the Koopman operator.

Recall from \Cref{sec:preliminaries} that $\Phi$ defines a flow on a state space $\Yc$, stemming, for instance, from the discretization of an ODE or PDE in time. Additionally, $\Yc$ comes with a probability measure $P$ that allows us to draw samples from it. To define the EDMD loss, we additionally require that the observables are square-integrable and indexed by a finite set $\Oc$, i.e., $\psi : \Yc \rightarrow \C^\Oc$ with $\psi_o, \Koopman \psi_o \in L^2_P(\Yc)$ for all $o \in \Oc$. Finally, recall that $\Gc$ denotes a finite group.

Our assumptions are then:

\begin{enumerate}[noitemsep, label=\normalfont{(\textbf{A\arabic*})}]
    \item $\Gc$ acts on $\Yc$ from the left, and it acts freely on $\Oc$ from the right.
    \label{ass:group-action}
    \item $\psi_o(g \cdot y) = \psi_{o \cdot g}(y)$ for $g \in \Gc, o \in \Oc$ and $y \in \Yc$.
    \label{ass:equivariance-observable}
    \item $\Phi(g \cdot y) = g \cdot \Phi(y)$ for $g \in \Gc$ and $y \in \Yc$. \label{ass:equivariance-dynamics}
    \item $\bm y \sim P$ implies $g \cdot \bm y \sim P$ for $g\in \Gc$. (Equivalently: $P(\Ac) = P(g \cdot \Ac)$ for $g \in \Gc, \Ac \in \Sigma$.) %
    \label{ass:invariance-density}
\end{enumerate}
Assumption \ref{ass:group-action} is very basic; it ensures that the group can act on both the system state and the observables. Intuitively, we need a group action on the system state to even consider group-based methods, and we need a group action on the index set of the observables so that the equivariance of the underlying system carries over to the level of observables. The next assumptions \ref{ass:equivariance-observable} and \ref{ass:equivariance-dynamics} are in the same spirit; the second ensures that the group operation can shift from state to observation level and the third assumes left equivariance of the system dynamics.\footnote{Note that assumption \ref{ass:equivariance-dynamics} already implies equivariance of the Koopman operator: One can define a group action on observables of the form $f : \Yc \rightarrow \C$ by $(g \cdot f)(y) \coloneqq f(g^{-1} \cdot y)$. The Koopman operator then satisfies
$
    \Koopman (g \cdot f) = g \cdot (\Koopman f)
$
(see Theorem III.1 in \cite{salovaKoopmanOperatorIts2019}).} Finally, \ref{ass:invariance-density} ensures that the underlying distribution is invariant with respect to group actions. That is, acting on states does not change the probability.

Our main result is now as follows:

\begin{theorem}
    \label{thr:convolution}
    Under \ref{ass:group-action}, \ref{ass:equivariance-observable}, \ref{ass:equivariance-dynamics} and \ref{ass:invariance-density}, %
    there is an equivariant EDMD matrix.%
\end{theorem}

Before going into the proof, note that the minimization problem \eqref{eq:edmd-loss} can have multiple solutions if no further restrictions are applied. However, under suitable conditions, the solution is unique (see, e.g., \cite[Lemma C.2]{PSB+24}). In that case, according to this theorem, the unique optimal EDMD matrix must be equivariant. 

\begin{proof}
    Introduce $\bm u = \psi(\bm y)$ and $\bm u^+ = \Koopman \psi(\bm y) $, and define 
    $
        \lossfull(K) = \mathbb E[\lVert K \bm u - \bm u^+ \rVert^2],
    $
    so that an EDMD matrix satisfies $K \in \argmin_{K \in \C^{\Oc\times \Oc}} \ell(K)$.
    It suffices to fix an arbitrary $K$ and to show that there is an equivariant matrix $K'$ such that $\lossfull(K) \geq \lossfull(K')$. 
    
    Let $\Sc \subseteq \Oc$ be a set of orbit representatives.
    Since $\Gc$ defines a free action on $\Oc$ by \ref{ass:group-action}, the mapping $(s,g) \mapsto s \cdot g$ constitutes a bijection between $\Sc \times \Gc$ and $\Oc$ (see also \Cref{rem:orbits-and-free-actions}). Hence,
    \begin{equation}
        \lossfull(K)
        = \sum_{g \in \Gc} \sum_{s \in \Sc} \mathbb E [( (K \bm u)_{s \cdot g} - \bm u^+_{s \cdot g} )^2].
        \label{eq:thr-convolution-optimality-lossfull}
    \end{equation}
    Select the group element that minimizes the inner sum:
    $$
        g_\star \in \argmin_{g \in \Gc} \sum_{s \in \Sc} \mathbb E [( (K \bm u)_{s \cdot g} - \bm u^+_{s \cdot g} )^2].
    $$
    We then define $K'$ by setting
    $$
        K'_{s \cdot g,s' \cdot h} \coloneqq K_{s \cdot g_\star, s' \cdot h g^{-1} g_\star},
    $$
    and note that it is indeed equivariant.
    Then, for arbitrary $g \in \Gc$, $s \in \Sc$ and $y \in \Yc$,
    \begin{align}
        (K' \psi(g \cdot y))_{s \cdot g_\star g^{-1}}
        &= \textstyle\sum_{(s',h) \in \Sc \times \Gc} 
            K'_{s \cdot g_\star g^{-1}, s' \cdot h} \psi_{s' \cdot h}(g \cdot y) 
            \nonumber \\
        &= \textstyle\sum_{(s',h) \in \Sc \times \Gc} 
            K'_{s \cdot g_\star, s' \cdot h g} \psi_{s' \cdot h g}(y) 
            \label{eq:thr-convolution-optimality-1} \\
        &= \textstyle\sum_{(s',h) \in \Sc \times \Gc} 
            K_{s \cdot g_\star, s' \cdot h g} \psi_{s' \cdot h g}(y) 
            \label{eq:thr-convolution-optimality-2} \\
        &= \textstyle\sum_{(s',h) \in \Sc \times \Gc} 
            K_{s \cdot g_\star, s' \cdot h} \psi_{s' \cdot h}(y) 
            \label{eq:thr-convolution-optimality-3} \\
        &= (K \psi(y))_{s \cdot g_\star}.
            \label{eq:thr-convolution-optimality-4}
    \end{align}
    \Cref{eq:thr-convolution-optimality-1} uses \ref{ass:equivariance-observable} and equivariance of $K'$. Then \eqref{eq:thr-convolution-optimality-2} follows by definition of $K'$. Finally, \eqref{eq:thr-convolution-optimality-3} holds by the substitution $h \mapsto h g^{-1}$.
    Hence,
    \begin{align}
        \lossfull(K) 
        &\geq 
            \textstyle\sum_{(s,g) \in \Sc \times \Gc} \mathbb E [((K \bm u )_{s \cdot g_\star} - \bm u^+_{s \cdot g_\star})^2] 
            \label{eq:thr-convolution-optimality-6}
            \\
        &=
            \textstyle\sum_{(s,g) \in \Sc \times \Gc} \mathbb E [((K'\psi(g \cdot \bm y))_{s \cdot (g_\star g^{-1})} - \psi_{s \cdot g_\star}(\Phi(\bm y)))^2] 
            \label{eq:thr-convolution-optimality-7}
            \\
        &=
            \textstyle\sum_{(s,g) \in \Sc \times \Gc} \mathbb E [((K' \psi(g \cdot \bm y))_{s \cdot (g_\star g^{-1})} - \psi_{s \cdot (g_\star g^{-1})}(\Phi(g \cdot \bm y)))^2] 
            \label{eq:thr-convolution-optimality-8}
            \\
        &=
            \textstyle\sum_{(s,g) \in \Sc \times \Gc} \mathbb E [((K'\bm u)_{s \cdot (g_\star g^{-1})} - \bm u^+_{s \cdot  (g_\star g^{-1})})^2 ]
            \label{eq:thr-convolution-optimality-9}
            \\
        &=
            \textstyle\sum_{(s,g) \in \Sc \times \Gc} \mathbb E [((K'\bm u)_{s \cdot g} - \bm u^+_{s \cdot  g})^2]
            = \lossfull(K').
            \nonumber
    \end{align}
    \Cref{eq:thr-convolution-optimality-6} is obtained when bounding the outer sum in \eqref{eq:thr-convolution-optimality-lossfull} from below by substituting $g \mapsto g_\star$. Then, \eqref{eq:thr-convolution-optimality-7} holds by \eqref{eq:thr-convolution-optimality-4} and definition of $\bm u^+$. Applying \ref{ass:equivariance-dynamics} and \ref{ass:equivariance-observable}, we get \eqref{eq:thr-convolution-optimality-8}. 
    Then we can apply the substitution $g \cdot \bm y \mapsto \bm y$ inside the expectation by \ref{ass:invariance-density}, and reintroduce $\bm u$ and $\bm u^+$. Finally, we reorder the summation using the substitution $g \mapsto g^{-1} g_\star$.
\end{proof}

\paragraph{Constructing observables}
In order to construct observables satisfying \ref{ass:group-action} and \ref{ass:equivariance-observable}, one can start with an arbitrary finite set $\Sc$ and a set of \emph{generating} observables $\chi :  \Yc \rightarrow \C^\Sc$ that are then expanded using the group. Letting $\Oc = \Sc \times \Gc$, we obtain a full set of observables by setting $\psi : \Yc \rightarrow \C^\Oc ,$ where $\psi_{(s,g)}( y ) = \chi_s( g \cdot y  )$. Indeed, $\Gc$ defines a free right-group action on $\Oc$ via $(s, h) \cdot g = (s, h g)$, and we see that it satisfies \ref{ass:equivariance-observable}:
$$
    \psi_{(s,h)}(g \cdot y) 
    = \chi_s( h \cdot (g \cdot y) ) 
    = \chi_s( (h g) \cdot y ) 
    = \psi_{(s, h g)}(y)
    = \psi_{(s, h) \cdot g}(y).
$$

This argument also works the other way around: We can obtain a set of generating observables given that $\psi$ satisfies \ref{ass:group-action} and \ref{ass:equivariance-observable}: In that case, since $\Gc$ acts freely on $\Oc$, there is a finite set of orbit-representatives $\Sc \subseteq \Oc$ such that $(s,g) \mapsto s \cdot g$ is a bijection between $\Sc \times \Gc$ and $\Oc$. Therefore, one can define a selection of generating observables by setting $\chi_s(y) = \psi_s(y)$, $s \in \Sc$, and recover $\psi_{s \cdot g}(y)$ using \ref{ass:equivariance-observable} and the identity $\psi_{s}(g \cdot y) = \chi_s(g \cdot y)$. (Constructing observables this way appears to be similar to that of \cite{salovaKoopmanOperatorIts2019}.)

\begin{example}
    \label{ex:generating-observable}
    Suppose $\Yc = \R^{\Z_n}$ is the set of functions from $\Z_n$ to $\R$. The group $\Gc = (\Z_n,+)$ can act on $\Yc$ by periodic shifts: For $g \in \Gc$ and $y \in \Yc$, define
    $$
        (g + y)_i \coloneqq y_{i-g},
    $$
    To get a generating observable, one may define $|\Sc| = 3$ and
    $$
        \chi(y) = \begin{bmatrix}
            y_0 & y_0^2 & y_0^3
        \end{bmatrix}.
    $$
    That is, $\chi(y)$ simply observes the state's value at the origin in various ways.
    The full set of observables $\psi : \Yc \rightarrow \C^\Oc$ is obtained when shifting the state periodically and then observing it at the origin by means of $\chi$. Practically, this has the same effect as observing the original state at different coordinates: Viewing $\psi(y)$ as a matrix in $\C^{|\Gc| \times |\Sc|}$:
    $$
        \psi(y) = \begin{bmatrix}
            \chi(y) \\ \chi( 1 + y ) \\ \chi( 2 + y ) \\ \vdots \\ \chi( (n-1) + y) 
        \end{bmatrix} 
        = \begin{bmatrix}
            y_0   & y_0^2 & y_0^3 \\
            y_{n-1} & y_{n-1}^2 & y_{n-1}^3 \\
            y_{n-2} & y_{n-2}^2 & y_{n-2}^3 \\
            \vdots & \vdots & \vdots \\
            y_{1} & y_1^2 & y_1^3
        \end{bmatrix}.
    $$
\end{example}

\paragraph{Training}
Instead of constructing a full equivariant EDMD matrix, we can compute its convolution kernel instead. That is, we can solve the problem
$$
    \min_{A : \Gc \rightarrow \C^{\Sc \times \Sc}} \mathbb E [ \lVert A * \bm u - \bm u^+ \rVert^2 ],
$$
where $\bm u$ and $\bm u^+$ are $\Gc \rightarrow \C^\Sc$-valued random variables with an arbitrary distribution. A convenient way is to do it in Fourier space. Using Plancherel's identity (\Cref{thr:plancherel}) and the convolution theorem \eqref{eq:convolution-theorem}, one has
$$
    \mathbb E [ \lVert A * \bm u - \bm u^+ \rVert^2 ] =
    \frac 1 {|\Gc|} \sum_{\rho \in \Fourier{\Gc}} d_\rho \mathbb E [ \lVert \Fourier{ A }(\rho) \Fourier{\bm u}(\rho) - \Fourier{\bm u^+}(\rho) \rVert^2_F ],
$$
where $\lVert \cdot \rVert_F$ is the Frobenius norm.
When minimizing the right-hand side, notice that the $\{ \Fourier{A}(\rho) : \rho \in \Fourier{\Gc} \}$ are independent. Hence, we can individually solve $|\Fourier{\Gc}|$ least-squares problems:
\begin{equation}
     \min_{\Fourier{A}(\rho) \in \C^{|\Sc| d_\rho \times |\Sc|d_\rho}} \mathbb E[ \lVert \Fourier{A}(\rho) \Fourier{\bm u}(\rho) - \Fourier{\bm u^+}(\rho) \rVert^2_F ],
     \label{eq:convolution-separated-optimization-problem}
\end{equation}
i.e., one for every representation. To recover $A$, one can apply the inverse Fourier transform.

\paragraph{Eigenfunctions}
One commonly approximates eigenfunctions of the Koopman operator by taking left eigenpairs $(\lambda , v)$ of the EDMD matrix, i.e., $v^T K  = \lambda v^T$, and introducing $f(y) = v^T \psi(y) = \sum_{o \in \Oc} v_o \psi_o(y)$ (see, e.g., \cite{WKR15}) to approximate an eigenfunction. Indeed, one has 
$$
    \Koopman f(y) = \Koopman (v^T \psi )(y) = v^T \Koopman \psi(y) \approx v^T K \psi(y) = \lambda v^T \psi(y) = \lambda f(y).
$$
If the EDMD matrix is equivariant, one can simply  compute eigenpairs of its ``transposed'' convolution kernel using the approach from \Cref{sec:convolutions-and-fourier-transform}.

\begin{table}[ht]
    \centering
    \begin{tabular}{c|cc}
        Parameter & Kuramoto--Sivashinsky & Spiraling waves \\
        \hline
        Time step & $1$ time unit & $0.2$ time units \\
        Spatial domain & $[0,3\pi] \times [0,3\pi]$ & $[0,40]\times[0,40]$ \\
        grid / group $\Gc$ & $\Z_{16} \times \Z_{16}$ & $\Z_{48} \times \Z_{48}$ \\
        $|\Sc|$ & 5 & 30 \\ 
        observable size $|\Gc||\Sc|$ & 1,280 & 69,120
    \end{tabular}
    \caption{Simulation settings.}
    \label{tab:simulation-details}
\end{table}

\subsection{Experiments: Partial-differential equations with periodic boundary conditions}\label{sec:experiments}

The goal of this section is to demonstrate the convolutional EDMD approach for the case of PDEs in two-dimensional periodic domains using the abelian group $\Gc = (\Z_n \times \Z_n,+)$ and the fast Fourier transform. We first show how to construct observables using the notion of generating observable, then demonstrate how to learn a convolution kernel in Fourier space, and finally show how to approximate eigenfunctions of the Koopman operator by computing eigenvectors of convolutions.\footnote{Code and detailed description available online, \url{https://github.com/graps1/convolutional-edmd}.}

Before going into details, let us remark on some methodical aspects. The error \eqref{eq:edmd-loss} is generally smaller if the EDMD matrix is left unconstrained, i.e., forcing it to be equivariant can only increase it. However, this is no longer true once we are concerned with finite-data approximations to the EDMD matrix. In this section, we will compare the convolutional with the full-scale EDMD approach---both refer to the case where \eqref{eq:edmd-loss} is approximated with finitely many samples.

There are two advantages that come with equivariant matrices: Firstly, the additional structure means that fewer parameters are learned, and thus, if the expectation in \eqref{eq:edmd-loss} is approximated using only few data points, an equivariant approach may compute better solutions. Secondly, if the group is large and if the Fourier transform can be computed efficiently, as is the case here, then convolutions may provide a computational advantage over using a full-scale EDMD matrix.

To test these claims, we conduct the following experiments:
\begin{itemize}
    \item To demonstrate that convolutional EDMD can learn even on sparse/incomplete data, we carry out one experiment in the ``low-data'' and one in the ``large-data'' regime using the Kuramoto--Sivashinsky equation 
    $$
        \dot y = - \Delta y - \Delta^2 y - \textstyle\frac 1 2  \lVert \nabla y \rVert^2.
    $$
    For each case, we learn both an equivariant EDMD matrix that is represented by a convolution kernel, and a full-scale EDMD matrix that is left unconstrained. 
    \item To show that the convolutional EDMD approach is applicable even in high-dimensional settings, we conduct an experiment on a $\lambda$-$\omega$ spiraling wave system (see \cite{aguareles2016asymptotic}, \cite{nathan2018applied}), parametrized as in \cite{walker2023visualpde},
    \begin{alignat*}{2}
        \dot u = &\quad\textstyle \frac 1 5 &\Delta u + 3 u + (v - u)(u^2 +v^2), \\
        \dot v = &&\Delta v + 3 v - (u+v)(u^2+ v^2) ,
    \end{alignat*}
    with $y = (u,v)$.
    The system has a limit cycle that depends on the initial conditions, see the top two rows in \Cref{fig:spiralwaves-comparison}. Since the number of observables is in the ten-thousands, using a full-scaler approach is infeasible. We show that the convolutional EDMD approach can still learn the system dynamics and can efficiently compute thousands of eigenvalues.
\end{itemize}

\paragraph{Setup}
Both demonstrations have the same baseline. Each PDE is solved and then discretized in both space and time, see \Cref{tab:simulation-details} for details. The spatial domain becomes a grid of shape $\Z_n \times \Z_n$, so that the states are of the form $y : \Z_n \times \Z_n \rightarrow \R^m, y \in \Yc$ ($m=1$ for Kuramoto--Sivashinsky and $m = 2$ for the spiraling waves equation). 
After temporal discretization, we obtain a dynamical system $\Phi : \Yc \rightarrow \Yc$. The translational equivariance of the PDE is inherited by this dynamical system, which together with the periodic boundary conditions yields equivariance of $\Phi$ in the sense that 
$$
    \Phi(g + y) = g + \Phi(y) \quad\text{for}\quad g \in \Gc = (\Z_n \times \Z_n, +), 
$$
where $g$'s action on $y$ is defined pointwise, i.e., $(g + y)(x) = y(x - g)$. In other words, shifting the states in space and propagating them yields the same result as shifting the propagated states.

\paragraph{Observables and reconstruction} Since $\Gc = \Z_n \times \Z_n$, one can view states $y : \Gc \rightarrow \R^m$ as $n \times n$ images with $m$ channels. In order to define observables matching the group structure, we employ a $1 \times 1$ convolutional neural network that expands $m$ to $|\Sc|$ channels. In fact, this architectural choice can also be derived using a generating observable of the form 
$$
    \chi : \Yc \rightarrow \R^\Sc, \quad \chi(y) = \sigma( A y(0)+b),
$$ 
where $A \in \R^{|\Sc|\times m}, b \in \R^{|\Sc|}$
are parameters and $\sigma$ is a pointwise nonlinearity. The convolutional network is then the observable $\psi : \Yc \rightarrow \R^\Oc$ that is ``generated'' by $\chi$ (see \Cref{sec:equivariance-edmd} and also \Cref{ex:generating-observable}). 
In all of our experiments, it was sufficient to initialize $A,b$ randomly and to use the softplus activation function for $\sigma$. For reconstruction, we use a $1 \times 1$ linear convolutional neural network mapping $\Gc \rightarrow \R^\Sc$ back to $\Gc \rightarrow \R^m$.

\paragraph{Training}
Following \Cref{sec:equivariance-edmd}, one can learn the optimal convolution kernel in Fourier space, which amounts to $|\Fourier{\Gc}| = |\Gc| = n^2$ small problems that need to be solved. 
Since $\Gc$ is abelian, the generalized Fourier transform is simply a two-dimensional discrete Fourier transform (see end of \Cref{sec:convolutions-and-fourier-transform}), with fast implementations widely available. Additionally, we solve the single large problem \eqref{eq:edmd-loss} for the full EDMD matrix for the case of the Kuramoto--Sivashinsky equation.

\paragraph{Eigenvalues and eigenfunctions}
Eigenvalues are computed using \Cref{thr:eigenvalues}, and eigenfunctions are approximated by computing eigenpairs of the ``transposed'' convolution. To demonstrate the correctness of the approximated eigenfunctions, one defines the \emph{squared relative residual} (SRR), see \cite{colbrook2023residual}:
\begin{equation}
    \mathrm{SRR}_{\mathrm{eig}}(\lambda, f) = 100\% \cdot \frac{\mathbb E_{\Ptest} [|\Koopman f(\bm y) - \lambda f(\bm y)|^2]}{\mathbb E_{\Ptest}[|f(\bm y)|^2]},
    \label{eq:error-eigenfunctions}
\end{equation}
where $f \in L^2_P(\Yc)$ and $\lambda \in \C$. %
We approximate the SRR by averaging over the test samples.

\begin{figure}[h!]
    \centering
    \includegraphics[width=0.45\linewidth]{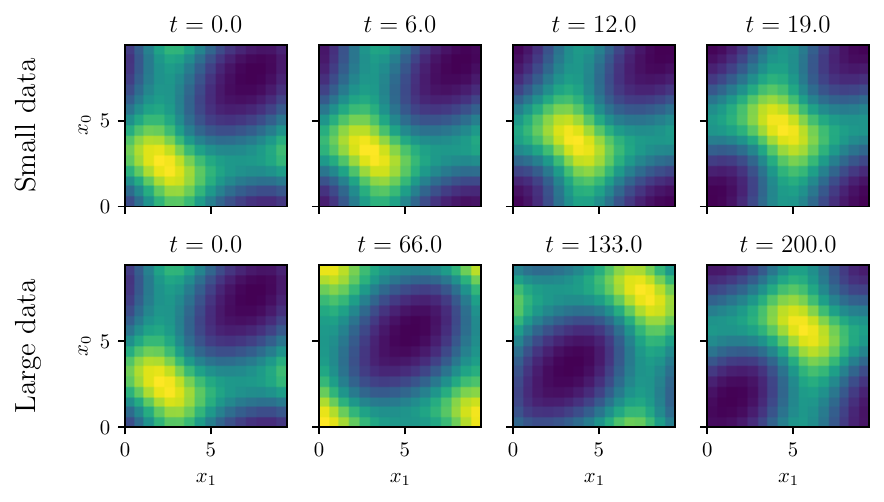}
    \includegraphics[width=0.45\linewidth]{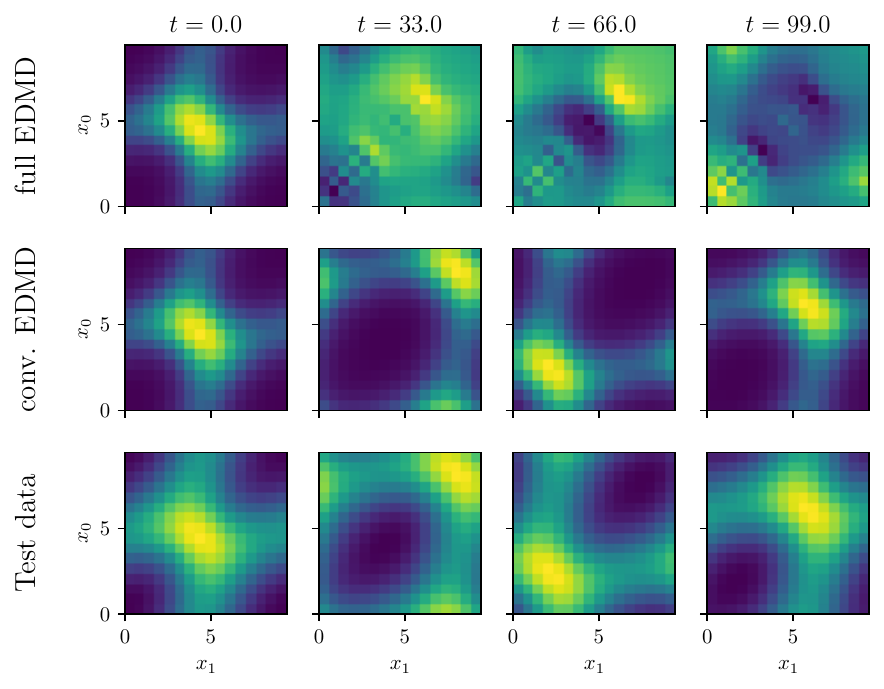}
    \caption{Left: Dataset used for the KS equation. The small dataset contains the system trajectory up until $t = 20$, the large dataset contains the system trajectory up until $t = 300$. Right: Comparison of the generated trajectories of the convolutional vs. full-scale EDMD approach on the test data.}
    \label{fig:ks2d-trajectories-comparison}
\end{figure}

\begin{figure}[h!]
    \centering
    \includegraphics[width=0.5\linewidth]{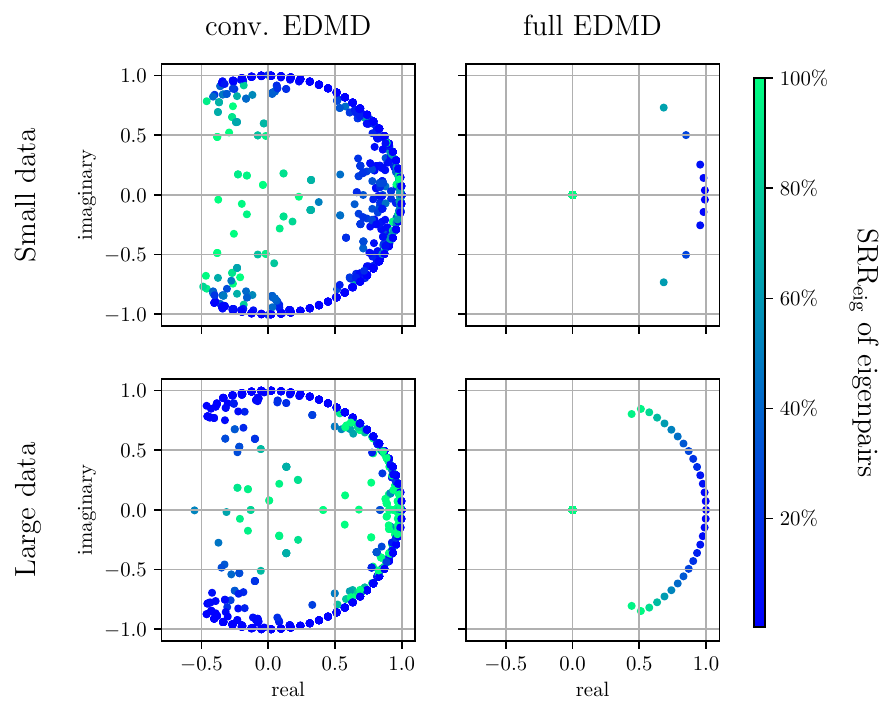}
    \includegraphics[width=0.3\linewidth]{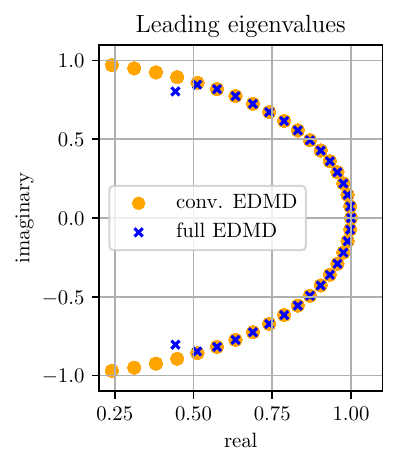}
    \caption{Left: Computed eigenvalues of both approaches on both datasets. Each eigenvalue is colored by $\mathrm{SRR}_{\mathrm{eig}}$, allowing us to evaluate the quality of the computed eigenpair (smaller is better). Right: comparison of the leading eigenvalues on the large dataset.}
    \label{fig:ks2d-eigenvalues}
\end{figure}

\subsubsection{Kuramoto--Sivashinsky equation: Performance on small vs. large datasets}

The dynamics of the Kuramoto--Sivashinsky equation are chaotic for larger domain sizes, but for our chosen initial condition and domain size, the solution tends to a wave that travels periodically from the lower left to the upper right corner (see \Cref{fig:ks2d-trajectories-comparison}). We simulate the system's trajectory until it settles onto its attractor, and record the trajectory for another 300 steps. For the low-data regime, we select only 20 input-output pairs from the training data, and for the large-data regime, we select all training samples. Another 200 separate samples make up the test data, which are used for evaluation.

In the \emph{low-data regime}, the convolutional approach is able to capture the system dynamics with small error, see the right side in \Cref{fig:ks2d-trajectories-comparison}, whereas the full-scale EDMD approach fails to learn. Remarkably, this dataset does not contain a complete cycle of the traveling wave. But due to the imposed structure, the convolutional model is still able to accurately describe the behavior of the traveling wave as it reaches the top right corner.
Furthermore, we can see that the convolutional approach captures most of the eigenvalue spectrum compared to the eigenvalues in the large-data regime, see \Cref{fig:ks2d-eigenvalues}. The leading eigenfunctions of the convolutional approach have a high quality in terms of the SRR and do not improve much with more data. 
In the \emph{large-data regime}, both approaches work well, but the convolutional EDMD approach has lower error, even though both performances are comparable to the convolutional approach in the low-data regime. (A figure such as the right one in \Cref{fig:ks2d-trajectories-comparison} for the large-data regime is left out for brevity since the error is not recognizable by eye.) Here, the leading eigenvalues of both approaches are in strong agreement.

\begin{figure}[h!]
    \centering
    \includegraphics[width=0.8\linewidth]{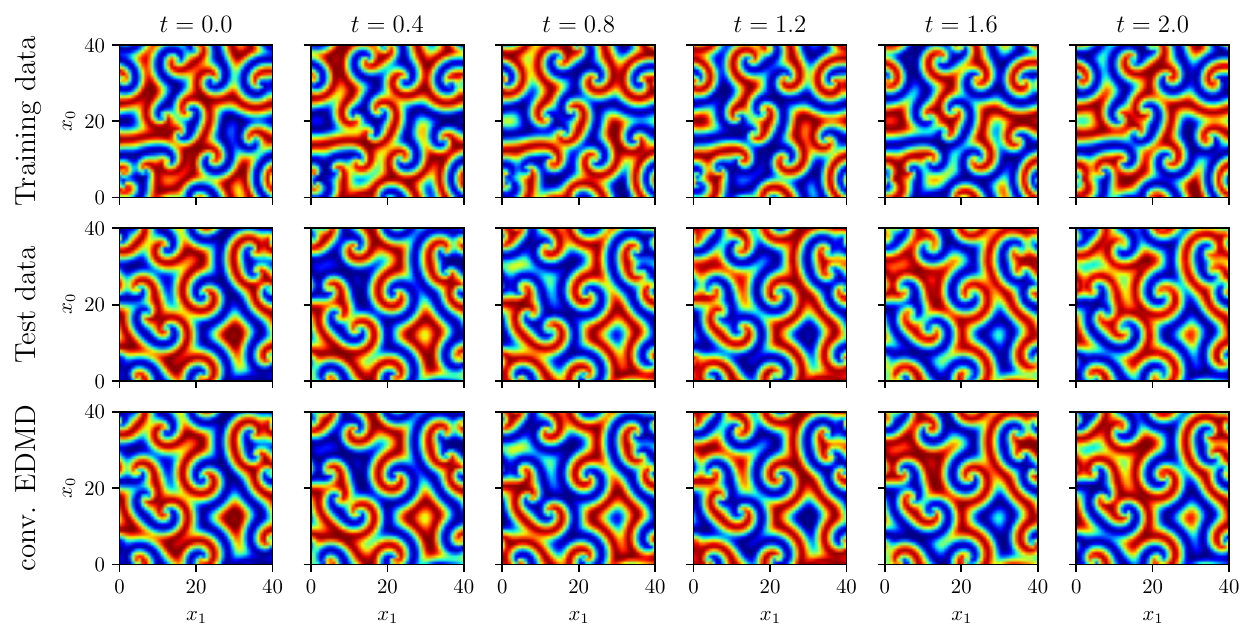}
    \caption{Top two rows: datasets generated by solving the spiraling waves equation starting from different initial conditions. Bottom row: trajectory generated by convolutional EDMD on the test data.}
    \label{fig:spiralwaves-comparison}
\end{figure}

\begin{figure}[h!]
    \centering
    \includegraphics[width=0.8\linewidth]{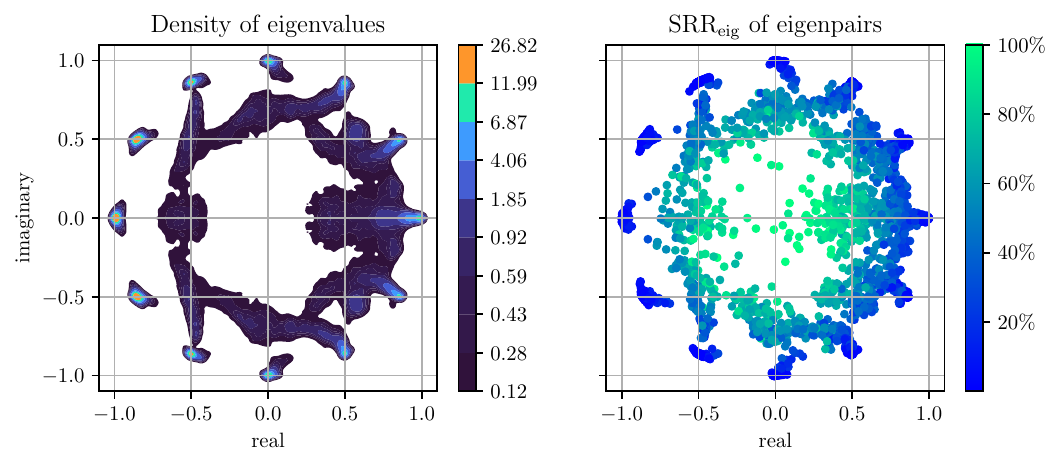}
    \caption{Left: Density plot of all eigenvalues (ca. 69k) computed by convolutional EDMD. Right: randomly sampled eigenvalues colored by $\mathrm{SRR}_{\mathrm{eig}}$.}
    \label{fig:spiralwaves-eigenvalues}
\end{figure}

\subsubsection{Spiraling waves equation: Performance on high-dimensional data}

For this equation, we have two datasets generated from different initial conditions, see \Cref{fig:spiralwaves-comparison}. Based on these initial conditions, the spiraling waves equation tends to a periodic limit cycle that features waves forming and collapsing. We train the convolutional EDMD approach on one dataset and evaluate it on the other, however, due to the large underlying group and number of used observables, a comparison with full-scale EDMD is not viable.  In \Cref{fig:spiralwaves-comparison}, we can see that the convolutional approach faithfully captures the system dynamics when run on initial conditions different from the ones it was being trained on. 

Furthermore, since the group and the number of observables are both relatively large, we are able to compute ca. 69k eigenvalues that cluster around equidistant points on the unit circle (see left side of \Cref{fig:spiralwaves-eigenvalues}). For these eigenvalues, we can use the corresponding eigenvectors to approximate eigenfunctions of the Koopman operator. As can be seen on the right in \Cref{fig:spiralwaves-eigenvalues}, especially eigenpairs with eigenvalues close to the unit circle tend to be accurate.

\section{Conclusion}

We have shown that enforcing group-equivariant structures in the EDMD matrix can be useful for a few reasons. For example, one needs less data since the number of free parameters is reduced, and one may improve computational speed by means of fast (generalized) Fourier transforms. %
In view of practical applications we have illustrated that a convolutional approach may provide an advantage in the low-data regime by conducting experiments on the two-dimensional Kuramoto--Sivashinsky equation, and we have shown that it is performant even in very high-dimensional settings using the sprialing waves equation. Indeed, as we have argued, using an equivariant structure is also well-justified in face of \Cref{thr:convolution}.
For future work, there are two main avenues that one can pursue: Firstly, one could investigate systems that are only ``semi''-equivariant, for example, equivariant PDEs that have Dirichlet or Neumann instead of periodic boundary conditions. Here, it is not yet clear whether methods using the Fourier transformation are applicable. Nevertheless, one can possibly enforce different structures that keep the number of parameters small and still allow for efficient computations (e.g., \cite{baddooPhysicsinformedDynamicMode2023}). Albeit not as efficient as a Fourier transform, one can then apply other methods that treat the underlying matrix only implicitly, such as the Arnoldi iteration.
For the second avenue, one could investigate finite-data error bounds: The improved convergence of a convolutional method (when varying the number of data points) over a full-EDMD approach is likely better in theory too, and not only in practice.

\section{Acknowledgments}

H.H.\ acknowledges financial support by the project ``SAIL: SustAInable Life-cycle of Intelligent Socio-Technical Systems'' (Grant ID NW21-059D), which is funded by the program ``Netzwerke 2021'' of the Ministry of Culture and Science of the State of Northrhine Westphalia, Germany. %
F.P.\ was funded by the Carl Zeiss Foundation within the project DeepTurb--Deep Learning in and from Turbulence. He was further supported by the free state of Thuringia and the German Federal Ministry of Education and Research (BMBF) within the project THInKI--Th\"uringer Hochschulinitiative für KI im Studium.
S.P.\ acknowledges support by European Union via the ERC Starting Grant ``KoOpeRaDE'' (ID 101161457).

\bibliographystyle{elsarticle-num-names} 
\bibliography{references}

\begin{thebibliography}{41}
\expandafter\ifx\csname natexlab\endcsname\relax\def\natexlab#1{#1}\fi
\providecommand{\url}[1]{\texttt{#1}}
\providecommand{\href}[2]{#2}
\providecommand{\path}[1]{#1}
\providecommand{\DOIprefix}{doi:}
\providecommand{\ArXivprefix}{arXiv:}
\providecommand{\URLprefix}{URL: }
\providecommand{\Pubmedprefix}{pmid:}
\providecommand{\doi}[1]{\href{http://dx.doi.org/#1}{\path{#1}}}
\providecommand{\Pubmed}[1]{\href{pmid:#1}{\path{#1}}}
\providecommand{\bibinfo}[2]{#2}
\ifx\xfnm\relax \def\xfnm[#1]{\unskip,\space#1}\fi
\bibitem[{Trefethen and Bau(2022)}]{trefethen2022numerical}
\bibinfo{author}{L.~N. Trefethen}, \bibinfo{author}{D.~Bau}, \bibinfo{title}{Numerical linear algebra}, \bibinfo{publisher}{SIAM}, \bibinfo{year}{2022}.
\bibitem[{Willems et~al.(2005)Willems, Rapisarda, Markovsky, and De~Moor}]{willems2005note}
\bibinfo{author}{J.~C. Willems}, \bibinfo{author}{P.~Rapisarda}, \bibinfo{author}{I.~Markovsky}, \bibinfo{author}{B.~L. De~Moor},
\newblock \bibinfo{title}{A note on persistency of excitation},
\newblock \bibinfo{journal}{Systems \& Control Letters} \bibinfo{volume}{54} (\bibinfo{year}{2005}) \bibinfo{pages}{325--329}.
\bibitem[{Van~Waarde et~al.(2020)Van~Waarde, De~Persis, Camlibel, and Tesi}]{van2020willems}
\bibinfo{author}{H.~J. Van~Waarde}, \bibinfo{author}{C.~De~Persis}, \bibinfo{author}{M.~K. Camlibel}, \bibinfo{author}{P.~Tesi},
\newblock \bibinfo{title}{Willems’ fundamental lemma for state-space systems and its extension to multiple datasets},
\newblock \bibinfo{journal}{IEEE Control Systems Letters} \bibinfo{volume}{4} (\bibinfo{year}{2020}) \bibinfo{pages}{602--607}.
\bibitem[{Koopman(1931)}]{Koo31}
\bibinfo{author}{B.~O. Koopman},
\newblock \bibinfo{title}{{Hamiltonian systems and transformation in Hilbert space}},
\newblock \bibinfo{journal}{Proceedings of the National Academy of Sciences} \bibinfo{volume}{17} (\bibinfo{year}{1931}) \bibinfo{pages}{315--318}.
\bibitem[{Rowley et~al.(2009)Rowley, Mezi{\'{c}}, Bagheri, Schlatter, and Henningson}]{RMB+09}
\bibinfo{author}{C.~W. Rowley}, \bibinfo{author}{I.~Mezi{\'{c}}}, \bibinfo{author}{S.~Bagheri}, \bibinfo{author}{P.~Schlatter}, \bibinfo{author}{D.~S. Henningson},
\newblock \bibinfo{title}{{Spectral analysis of nonlinear flows}},
\newblock \bibinfo{journal}{Journal of Fluid Mechanics} \bibinfo{volume}{641} (\bibinfo{year}{2009}) \bibinfo{pages}{115--127}.
\bibitem[{Brunton et~al.(2022)Brunton, Budi{\v{s}}i{\'c}, Kaiser, and Kutz}]{brunton2022modern}
\bibinfo{author}{S.~L. Brunton}, \bibinfo{author}{M.~Budi{\v{s}}i{\'c}}, \bibinfo{author}{E.~Kaiser}, \bibinfo{author}{J.~N. Kutz},
\newblock \bibinfo{title}{Modern koopman theory for dynamical systems},
\newblock \bibinfo{journal}{SIAM Review} \bibinfo{volume}{64} (\bibinfo{year}{2022}) \bibinfo{pages}{229--340}.
\bibitem[{Mauroy et~al.(2020)Mauroy, Susuki, and Mezic}]{mauroy2020koopman}
\bibinfo{author}{A.~Mauroy}, \bibinfo{author}{Y.~Susuki}, \bibinfo{author}{I.~Mezic}, \bibinfo{title}{Koopman operator in systems and control}, volume \bibinfo{volume}{484}, \bibinfo{publisher}{Springer}, \bibinfo{year}{2020}.
\bibitem[{Bronstein et~al.(2021)Bronstein, Bruna, Cohen, and Veli{\v{c}}kovi{\'c}}]{BBCV21}
\bibinfo{author}{M.~M. Bronstein}, \bibinfo{author}{J.~Bruna}, \bibinfo{author}{T.~Cohen}, \bibinfo{author}{P.~Veli{\v{c}}kovi{\'c}}, \bibinfo{title}{{Geometric deep learning: Grids, groups, graphs, geodesics, and gauges}}, \bibinfo{year}{2021}. \bibinfo{note}{ArXiv:2104.13478}.
\bibitem[{Masci et~al.(2015)Masci, Boscaini, Bronstein, and Vandergheynst}]{MBB+15}
\bibinfo{author}{J.~Masci}, \bibinfo{author}{D.~Boscaini}, \bibinfo{author}{M.~M. Bronstein}, \bibinfo{author}{P.~Vandergheynst},
\newblock \bibinfo{title}{Geodesic convolutional neural networks on {R}iemannian manifolds},
\newblock in: \bibinfo{booktitle}{IEEE International Conference on Computer Vision Workshops}, \bibinfo{year}{2015}, pp. \bibinfo{pages}{37--45}.
\bibitem[{Cohen and Welling(2016)}]{CW16}
\bibinfo{author}{T.~S. Cohen}, \bibinfo{author}{M.~Welling},
\newblock \bibinfo{title}{Group equivariant convolutional networks},
\newblock in: \bibinfo{booktitle}{International Conference on Machine Learning}, volume~\bibinfo{volume}{33}, \bibinfo{year}{2016}, pp. \bibinfo{pages}{2990--2999}.
\bibitem[{Bronstein et~al.(2017)Bronstein, Bruna, LeCun, Szlam, and Vandergheynst}]{BBL+17}
\bibinfo{author}{M.~M. Bronstein}, \bibinfo{author}{J.~Bruna}, \bibinfo{author}{Y.~LeCun}, \bibinfo{author}{A.~Szlam}, \bibinfo{author}{P.~Vandergheynst},
\newblock \bibinfo{title}{Geometric deep learning: {G}oing beyond {E}uclidean data},
\newblock \bibinfo{journal}{IEEE Signal Processing Magazine} \bibinfo{volume}{34} (\bibinfo{year}{2017}) \bibinfo{pages}{18--42}.
\bibitem[{Atz et~al.(2021)Atz, Grisoni, and Schneider}]{AGS21}
\bibinfo{author}{K.~Atz}, \bibinfo{author}{F.~Grisoni}, \bibinfo{author}{G.~Schneider},
\newblock \bibinfo{title}{Geometric deep learning on molecular representations},
\newblock \bibinfo{journal}{Nature Machine Intelligence} \bibinfo{volume}{3} (\bibinfo{year}{2021}) \bibinfo{pages}{1023--1032}.
\bibitem[{Harder et~al.(2024)Harder, Rabault, Vinuesa, Mortensen, and Peitz}]{HRV+24}
\bibinfo{author}{H.~Harder}, \bibinfo{author}{J.~Rabault}, \bibinfo{author}{R.~Vinuesa}, \bibinfo{author}{M.~Mortensen}, \bibinfo{author}{S.~Peitz},
\newblock \bibinfo{title}{Solving partial differential equations with equivariant extreme learning machines},
\newblock in: \bibinfo{booktitle}{Symposium on Systems Theory in Data and Optimization (SysDO)}, \bibinfo{year}{2024}. \bibinfo{note}{(Accepted; preprint: arXiv:2404.18530)}.
\bibitem[{Dierkes et~al.(2023)Dierkes, Offen, Ober-Bl{\"o}baum, and Fla{\ss}kamp}]{dierkes2023hamiltonian}
\bibinfo{author}{E.~Dierkes}, \bibinfo{author}{C.~Offen}, \bibinfo{author}{S.~Ober-Bl{\"o}baum}, \bibinfo{author}{K.~Fla{\ss}kamp},
\newblock \bibinfo{title}{Hamiltonian neural networks with automatic symmetry detection},
\newblock \bibinfo{journal}{Chaos: An Interdisciplinary Journal of Nonlinear Science} \bibinfo{volume}{33} (\bibinfo{year}{2023}).
\bibitem[{Sinha et~al.(2020)Sinha, Nandanoori, and Yeung}]{sinha2020koopman}
\bibinfo{author}{S.~Sinha}, \bibinfo{author}{S.~Nandanoori}, \bibinfo{author}{E.~Yeung},
\newblock \bibinfo{title}{Koopman operator methods for global phase space exploration of equivariant dynamical systems},
\newblock \bibinfo{journal}{IFAC-PapersOnLine} \bibinfo{volume}{53} (\bibinfo{year}{2020}) \bibinfo{pages}{1150--1155}.
\bibitem[{Peitz et~al.(2024)Peitz, Harder, Nüske, Philipp, Schaller, and Worthmann}]{PHN+24}
\bibinfo{author}{S.~Peitz}, \bibinfo{author}{H.~Harder}, \bibinfo{author}{F.~Nüske}, \bibinfo{author}{F.~Philipp}, \bibinfo{author}{M.~Schaller}, \bibinfo{author}{K.~Worthmann},
\newblock \bibinfo{title}{{Equivariance and partial observations in Koopman operator theory for partial differential equations}},
\newblock \bibinfo{journal}{Journal of Computational Dynamics}  (\bibinfo{year}{2024}). \bibinfo{note}{(Accepted; preprint: arXiv:2307.15325)}.
\bibitem[{Salova et~al.(2019)Salova, Emenheiser, Rupe, Crutchfield, and D'Souza}]{salovaKoopmanOperatorIts2019}
\bibinfo{author}{A.~Salova}, \bibinfo{author}{J.~Emenheiser}, \bibinfo{author}{A.~Rupe}, \bibinfo{author}{J.~P. Crutchfield}, \bibinfo{author}{R.~M. D'Souza},
\newblock \bibinfo{title}{Koopman operator and its approximations for systems with symmetries},
\newblock \bibinfo{journal}{Chaos: An Interdisciplinary Journal of Nonlinear Science} \bibinfo{volume}{29} (\bibinfo{year}{2019}) \bibinfo{pages}{093128}.
\bibitem[{Baddoo et~al.(2023)Baddoo, Herrmann, McKeon, Nathan~Kutz, and Brunton}]{baddooPhysicsinformedDynamicMode2023}
\bibinfo{author}{P.~J. Baddoo}, \bibinfo{author}{B.~Herrmann}, \bibinfo{author}{B.~J. McKeon}, \bibinfo{author}{J.~Nathan~Kutz}, \bibinfo{author}{S.~L. Brunton},
\newblock \bibinfo{title}{Physics-informed dynamic mode decomposition},
\newblock \bibinfo{journal}{Proceedings of the Royal Society A} \bibinfo{volume}{479} (\bibinfo{year}{2023}) \bibinfo{pages}{20220576}.
\bibitem[{Hochrainer and Kar(2024)}]{hochrainerApproximationTranslationInvariant2024b}
\bibinfo{author}{T.~Hochrainer}, \bibinfo{author}{G.~Kar},
\newblock \bibinfo{title}{Approximation of translation invariant {{Koopman}} operators for coupled non-linear systems},
\newblock \bibinfo{journal}{Chaos: An Interdisciplinary Journal of Nonlinear Science} \bibinfo{volume}{34} (\bibinfo{year}{2024}) \bibinfo{pages}{083119}.
\bibitem[{Kutz et~al.(2018)Kutz, Proctor, and Brunton}]{nathan2018applied}
\bibinfo{author}{N.~J. Kutz}, \bibinfo{author}{J.~L. Proctor}, \bibinfo{author}{S.~L. Brunton},
\newblock \bibinfo{title}{Applied koopman theory for partial differential equations and data-driven modeling of spatio-temporal systems},
\newblock \bibinfo{journal}{Complexity} \bibinfo{volume}{2018} (\bibinfo{year}{2018}) \bibinfo{pages}{6010634}.
\bibitem[{Williams et~al.(2015)Williams, Kevrekidis, and Rowley}]{WKR15}
\bibinfo{author}{M.~O. Williams}, \bibinfo{author}{I.~G. Kevrekidis}, \bibinfo{author}{C.~W. Rowley},
\newblock \bibinfo{title}{{A data--driven approximation of the Koopman operator: Extending dynamic mode decomposition}},
\newblock \bibinfo{journal}{Journal of Nonlinear Science} \bibinfo{volume}{25} (\bibinfo{year}{2015}) \bibinfo{pages}{1307--1346}.
\bibitem[{Korda and Mezi{\'{c}}(2018)}]{KM18b}
\bibinfo{author}{M.~Korda}, \bibinfo{author}{I.~Mezi{\'{c}}},
\newblock \bibinfo{title}{On convergence of extended dynamic mode decomposition to the {K}oopman operator},
\newblock \bibinfo{journal}{Journal of Nonlinear Science} \bibinfo{volume}{28} (\bibinfo{year}{2018}) \bibinfo{pages}{687--710}.
\bibitem[{Klus et~al.(2018)Klus, Schuster, and Sch{\"u}tte}]{klus2016convergence}
\bibinfo{author}{S.~Klus}, \bibinfo{author}{I.~Schuster}, \bibinfo{author}{C.~Sch{\"u}tte},
\newblock \bibinfo{title}{On the convergence of extended dynamic mode decomposition to the {K}oopman operator},
\newblock \bibinfo{journal}{Journal of Nonlinear Science} \bibinfo{volume}{28} (\bibinfo{year}{2018}) \bibinfo{pages}{995--1010}. \DOIprefix\doi{10.1007/s00332-017-9437-9}.
\bibitem[{Mollenhauer et~al.(2022)Mollenhauer, Klus, Sch{\"u}tte, and Koltai}]{mollenhauer2022kernel}
\bibinfo{author}{M.~Mollenhauer}, \bibinfo{author}{S.~Klus}, \bibinfo{author}{C.~Sch{\"u}tte}, \bibinfo{author}{P.~Koltai},
\newblock \bibinfo{title}{Kernel autocovariance operators of stationary processes: Estimation and convergence},
\newblock \bibinfo{journal}{Journal of Machine Learning Research} \bibinfo{volume}{23} (\bibinfo{year}{2022}) \bibinfo{pages}{1--34}.
\bibitem[{Kostic et~al.(2024{\natexlab{a}})Kostic, Lounici, Novelli, and Pontil}]{kostic2024sharp}
\bibinfo{author}{V.~Kostic}, \bibinfo{author}{K.~Lounici}, \bibinfo{author}{P.~Novelli}, \bibinfo{author}{M.~Pontil},
\newblock \bibinfo{title}{Sharp spectral rates for {K}oopman operator learning},
\newblock \bibinfo{journal}{Advances in Neural Information Processing Systems} \bibinfo{volume}{36} (\bibinfo{year}{2024}{\natexlab{a}}).
\bibitem[{Kostic et~al.(2024{\natexlab{b}})Kostic, Inzerili, Lounici, Novelli, and Pontil}]{kostic2024consistent}
\bibinfo{author}{V.~Kostic}, \bibinfo{author}{P.~Inzerili}, \bibinfo{author}{K.~Lounici}, \bibinfo{author}{P.~Novelli}, \bibinfo{author}{M.~Pontil},
\newblock \bibinfo{title}{Consistent long-term forecasting of ergodic dynamical systems},
\newblock in: \bibinfo{booktitle}{2024 Int. Conf. on Machine Learning}, \bibinfo{year}{2024}{\natexlab{b}}.
\bibitem[{Mezić(2022)}]{Mez22}
\bibinfo{author}{I.~Mezić},
\newblock \bibinfo{title}{{On numerical approximations of the Koopman operator}},
\newblock \bibinfo{journal}{Mathematics} \bibinfo{volume}{10} (\bibinfo{year}{2022}) \bibinfo{pages}{1180}.
\bibitem[{Zhang and Zuazua(2023)}]{ZZ22}
\bibinfo{author}{C.~Zhang}, \bibinfo{author}{E.~Zuazua},
\newblock \bibinfo{title}{{A quantitative analysis of Koopman operator methods for system identification and predictions}},
\newblock \bibinfo{journal}{Comptes Rendus. M{\'{e}}canique} \bibinfo{volume}{351} (\bibinfo{year}{2023}) \bibinfo{pages}{1--31}.
\bibitem[{N{\"{u}}ske et~al.(2023)N{\"{u}}ske, Peitz, Philipp, Schaller, and Worthmann}]{NPP+23}
\bibinfo{author}{F.~N{\"{u}}ske}, \bibinfo{author}{S.~Peitz}, \bibinfo{author}{F.~Philipp}, \bibinfo{author}{M.~Schaller}, \bibinfo{author}{K.~Worthmann},
\newblock \bibinfo{title}{{Finite-data error bounds for Koopman-based prediction and control}},
\newblock \bibinfo{journal}{Journal of Nonlinear Science} \bibinfo{volume}{33} (\bibinfo{year}{2023}) \bibinfo{pages}{14}.
\bibitem[{Philipp et~al.(2024)Philipp, Schaller, Worthmann, Peitz, and N{\"{u}}ske}]{PSW+24}
\bibinfo{author}{F.~Philipp}, \bibinfo{author}{M.~Schaller}, \bibinfo{author}{K.~Worthmann}, \bibinfo{author}{S.~Peitz}, \bibinfo{author}{F.~N{\"{u}}ske},
\newblock \bibinfo{title}{{Error bounds for kernel-based approximations of the Koopman operator}},
\newblock \bibinfo{journal}{Applied and Computational Harmonic Analysis} \bibinfo{volume}{71} (\bibinfo{year}{2024}) \bibinfo{pages}{101657}.
\bibitem[{Llamazares-Elias et~al.(2024)Llamazares-Elias, Llamazares-Elias, Latz, and Klus}]{llamazares2024data}
\bibinfo{author}{L.~Llamazares-Elias}, \bibinfo{author}{S.~Llamazares-Elias}, \bibinfo{author}{J.~Latz}, \bibinfo{author}{S.~Klus},
\newblock \bibinfo{title}{Data-driven approximation of {K}oopman operators and generators: Convergence rates and error bounds},
\newblock \bibinfo{journal}{arXiv preprint arXiv:2405.00539}  (\bibinfo{year}{2024}).
\bibitem[{Köhne et~al.(2025)Köhne, Philipp, Schaller, { Schiela}, and Worthmann}]{KPS+24}
\bibinfo{author}{F.~Köhne}, \bibinfo{author}{F.~M. Philipp}, \bibinfo{author}{M.~Schaller}, \bibinfo{author}{A.~{ Schiela}}, \bibinfo{author}{K.~Worthmann},
\newblock \bibinfo{title}{${L}^\infty$-error bounds for approximations of the {K}oopman operator by kernel extended dynamic mode decomposition},
\newblock \bibinfo{journal}{SIAM Journal of Applied Dynamical Systems} \bibinfo{volume}{24} (\bibinfo{year}{2025}) \bibinfo{pages}{501--529}.
\bibitem[{Yadav and Mauroy(2025)}]{yadav2025approximation}
\bibinfo{author}{R.~Yadav}, \bibinfo{author}{A.~Mauroy},
\newblock \bibinfo{title}{Approximation of the {K}oopman operator via {B}ernstein polynomials},
\newblock \bibinfo{journal}{Communications in Nonlinear Science and Numerical Simulation}  (\bibinfo{year}{2025}) \bibinfo{pages}{108819}.
\bibitem[{Philipp et~al.(2024)Philipp, Schaller, Boshoff, Peitz, N{\"u}ske, and Worthmann}]{PSB+24}
\bibinfo{author}{F.~M. Philipp}, \bibinfo{author}{M.~Schaller}, \bibinfo{author}{S.~Boshoff}, \bibinfo{author}{S.~Peitz}, \bibinfo{author}{F.~N{\"u}ske}, \bibinfo{author}{K.~Worthmann},
\newblock \bibinfo{title}{Variance representations and convergence rates for data-driven approximations of koopman operators},
\newblock \bibinfo{journal}{arXiv preprint arXiv:2402.02494}  (\bibinfo{year}{2024}).
\bibitem[{{\AA}hlander and {Munthe-Kaas}(2005)}]{ahlanderApplicationsGeneralizedFourier2005}
\bibinfo{author}{K.~{\AA}hlander}, \bibinfo{author}{H.~{Munthe-Kaas}},
\newblock \bibinfo{title}{Applications of the {{Generalized Fourier Transform}} in {{Numerical Linear Algebra}}},
\newblock \bibinfo{journal}{BIT Numerical Mathematics} \bibinfo{volume}{45} (\bibinfo{year}{2005}) \bibinfo{pages}{819--850}.
\bibitem[{{\AA}hlander and Munthe-Kaas(2006)}]{ahlander_eigenvalues_2006}
\bibinfo{author}{K.~{\AA}hlander}, \bibinfo{author}{H.~Munthe-Kaas},
\newblock \bibinfo{title}{Eigenvalues for equivariant matrices},
\newblock \bibinfo{journal}{Journal of Computational and Applied Mathematics} \bibinfo{volume}{192} (\bibinfo{year}{2006}) \bibinfo{pages}{89--99}.
\bibitem[{Terras(1999)}]{terras1999fourier}
\bibinfo{author}{A.~Terras}, \bibinfo{title}{Fourier analysis on finite groups and applications}, number~\bibinfo{number}{43} in \bibinfo{series}{London Mathematical Society Student Texts}, \bibinfo{publisher}{Cambridge University Press}, \bibinfo{year}{1999}.
\bibitem[{Maslen and Rockmore(2001)}]{maslen2001cooley}
\bibinfo{author}{D.~K. Maslen}, \bibinfo{author}{D.~N. Rockmore},
\newblock \bibinfo{title}{{The Cooley-Tukey FFT and group theory}},
\newblock \bibinfo{journal}{Notices of the AMS} \bibinfo{volume}{48} (\bibinfo{year}{2001}) \bibinfo{pages}{1151--1160}.
\bibitem[{Aguareles et~al.(2016)Aguareles, Baldomá, and M-Seara}]{aguareles2016asymptotic}
\bibinfo{author}{M.~Aguareles}, \bibinfo{author}{I.~Baldomá}, \bibinfo{author}{T.~M-Seara},
\newblock \bibinfo{title}{On the asymptotic wavenumber of spiral waves in $\lambda-\omega$ systems},
\newblock \bibinfo{journal}{Nonlinearity} \bibinfo{volume}{30} (\bibinfo{year}{2016}). \DOIprefix\doi{10.1088/1361-6544/30/1/90}.
\bibitem[{Walker et~al.(2023)Walker, Townsend, Chudasama, and Krause}]{walker2023visualpde}
\bibinfo{author}{B.~J. Walker}, \bibinfo{author}{A.~K. Townsend}, \bibinfo{author}{A.~K. Chudasama}, \bibinfo{author}{A.~L. Krause},
\newblock \bibinfo{title}{Visualpde: rapid interactive simulations of partial differential equations},
\newblock \bibinfo{journal}{Bulletin of Mathematical Biology} \bibinfo{volume}{85} (\bibinfo{year}{2023}) \bibinfo{pages}{113}. \DOIprefix\doi{https://doi.org/10.1007/s11538-023-01218-4}.
\bibitem[{Colbrook et~al.(2023)Colbrook, Ayton, and Sz{\H o}ke}]{colbrook2023residual}
\bibinfo{author}{M.~J. Colbrook}, \bibinfo{author}{L.~J. Ayton}, \bibinfo{author}{M.~Sz{\H o}ke},
\newblock \bibinfo{title}{Residual dynamic mode decomposition: robust and verified {K}oopmanism},
\newblock \bibinfo{journal}{Journal of Fluid Mechanics} \bibinfo{volume}{955} (\bibinfo{year}{2023}) \bibinfo{pages}{A21}.

\end{thebibliography}

\end{document}